\documentclass[12pt]{amsart}
\usepackage[pdftex,colorlinks,citecolor=blue,urlcolor=blue]{hyperref}
\usepackage[margin=1in]{geometry}
\usepackage{tikz}
\usetikzlibrary{positioning, shapes.geometric, quotes, arrows}
\usepackage{amssymb}
\usepackage{subfig}
\usepackage{graphicx,color}
\usepackage{amsmath}
\usepackage{mathtools}
\usepackage{comment}
\usepackage{bbm}
\usepackage{MnSymbol}
\usepackage{cleveref}
\usepackage{todonotes}
\usepackage{caption}

\usepackage{thmtools}

\newtheorem{theorem}{Theorem}[section]
\newtheorem{proposition}[theorem]{Proposition}
\newtheorem{lemma}[theorem]{Lemma}
\newtheorem{corollary}[theorem]{Corollary}

\theoremstyle{definition}
\newtheorem{definition}[theorem]{Definition}
\newtheorem{defthm}[theorem]{Definition/Theorem}
\newtheorem{notation}[theorem]{Notation}
\newtheorem{example}[theorem]{Example}

\newtheorem{question}[theorem]{Question}

\newtheorem{remark}[theorem]{Remark}
\newtheorem{notationremark}[theorem]{Notation/Remark}

\newcommand{\Pc}{\mathcal{P}}
\newcommand{\Bc}{\mathcal{B}}
\newcommand{\Cc}{\mathcal{C}}
\newcommand{\Ck}{\mathfrak{C}}
\newcommand{\Bk}{\mathfrak{B}}
\newcommand{\cM}{\mathcal{M}}
\newcommand{\cf}{\mathbf{c}}
\newcommand{\ef}{\mathbf{e}}
\newcommand{\uf}{\mathbf{u}}
\newcommand{\ufp}{\mathbf{u}\boldsymbol{'}}
\newcommand{\vf}{\mathbf{v}}
\newcommand{\xf}{\mathbf{x}}
\newcommand{\yf}{\mathbf{y}}
\newcommand{\origin}{\mathbf{0}}
\newcommand{\NN}{\mathbb{N}} 
\newcommand{\ZZ}{\mathbb{Z}} 
 
\newcommand{\RR}{\mathbb{R}} 

\newcommand{\tlambda}{\tilde{\lambda}} 
 
\newcommand{\blambda}{\boldsymbol{\lambda}} 
 
\newcommand{\btlambda}{\boldsymbol{\tilde{\lambda}}} 
\newcommand{\btmu}{\boldsymbol{\tilde{\mu}}} 
\newcommand{\conv}{\mathop{\rm conv}\nolimits}
\newcommand{\supp}{\mathop{\rm supp}\nolimits}
\newcommand{\init}{\mathop{\rm in}\nolimits}
\newcommand{\sgncirc}{\overrightarrow{\mathfrak{C}}}

\newcommand{\Mpm}{\begin{bmatrix}M \ | \ -M\end{bmatrix}} 
\newcommand{\Npm}{\begin{bmatrix}N \ | \ -N\end{bmatrix}} 
\newcommand{\FMpm}{\begin{bmatrix}FM \ | \ -FM\end{bmatrix}}

\renewcommand\labelenumi{(\roman{enumi})}
\renewcommand\theenumi\labelenumi

\title[On a generalization of symmetric edge polytopes to regular matroids]{On a generalization of symmetric edge polytopes to regular matroids}

\author{Alessio D'Al\`{i}}
\address{Dipartimento di Matematica, Politecnico di Milano, Italy (current) and Department of Mathematics, University of Osnabr\"{u}ck, Germany}
\email{alessio.dali@polimi.it}
\author{Martina Juhnke-Kubitzke}
\address{Department of Mathematics, University of Osnabr\"{u}ck, Germany}
\email{juhnke-kubitzke@uni-osnabrueck.de}

\author{Melissa Koch}
\address{Department of Computer Science, University of Osnabr\"{u}ck, Germany}
\email{melikoch@uni-osnabrueck.de}
\begin{document}


\begin{abstract}Starting from any finite simple graph, one can build a reflexive polytope known as a symmetric edge polytope. The first goal of this paper is to show that symmetric edge polytopes are intrinsically matroidal objects: more precisely, we prove that two symmetric edge polytopes are unimodularly equivalent precisely when they share the same graphical matroid. The second goal is to show that one can construct a generalized symmetric edge polytope starting from every \emph{regular} matroid. Just like in the usual case, we are able to find combinatorial ways to describe the facets and an explicit regular unimodular triangulation of any such polytope. Finally, we show that the Ehrhart theory of the polar of a given generalized symmetric edge polytope is tightly linked to the structure of the lattice of flows of the dual regular matroid.
\end{abstract}

\maketitle

\section{Introduction}
Symmetric edge polytopes are a class of centrally symmetric reflexive lattice polytopes which has seen a lot of interest in the last few years due to their fascinating combinatorial properties \cite{HJM, OhsugiTsuchiya, OhsugiTsuchiyaNew, BraunBruegge, BraunBrueggeKahle, ChenDavisKorchevskaia, KaTo, DJKKV} and their connections to various branches of mathematics and physics \cite{ChenDavisMehta, DelucchiHoessly, CelikEtAl, DDM}.

Given a finite simple graph $G$ on vertex set $V = [n] \coloneqq \{1, 2, \ldots, n\}$, the \emph{symmetric edge polytope} associated with $G$ is the lattice polytope
\[\Pc_G \coloneqq \conv\{\pm(\ef_i - \ef_j) \mid \{i,j\} \in E(G)\} \subseteq \RR^{|V|},\]

where $\ef_i$ denotes the $i$-th standard basis vector.

\begin{figure}[h!]
\includegraphics[width=0.4\textwidth]{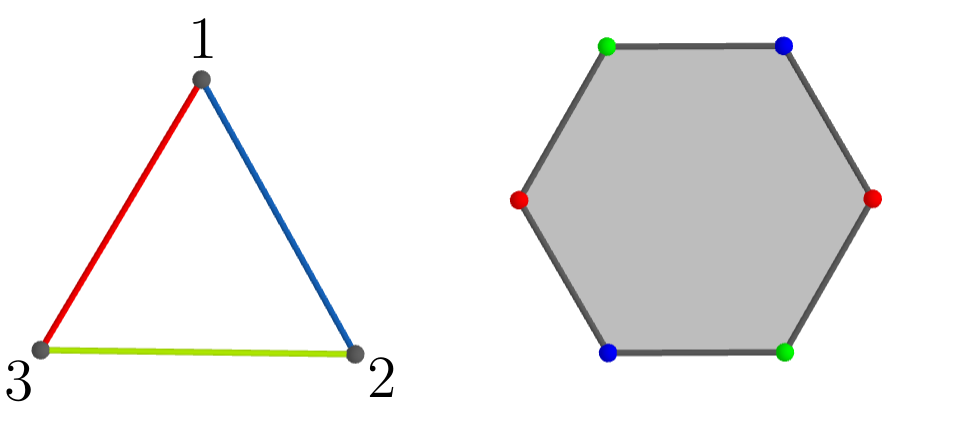}
\includegraphics[width=0.4\textwidth]{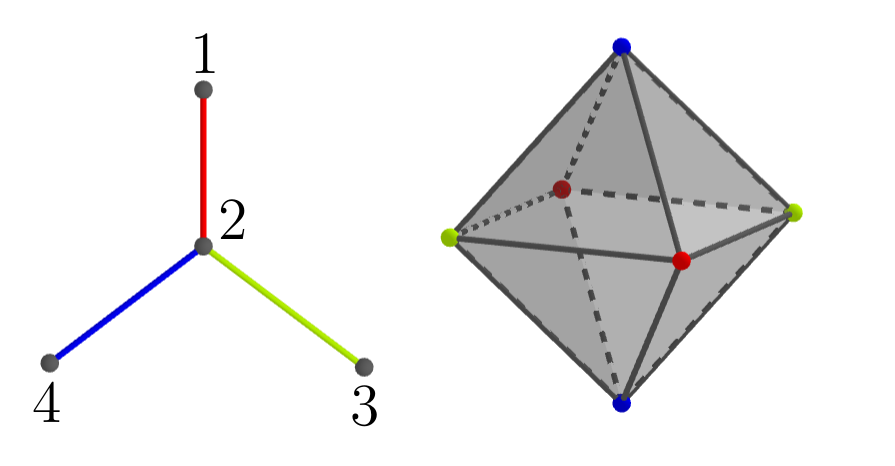}
\includegraphics[width=0.4\textwidth]{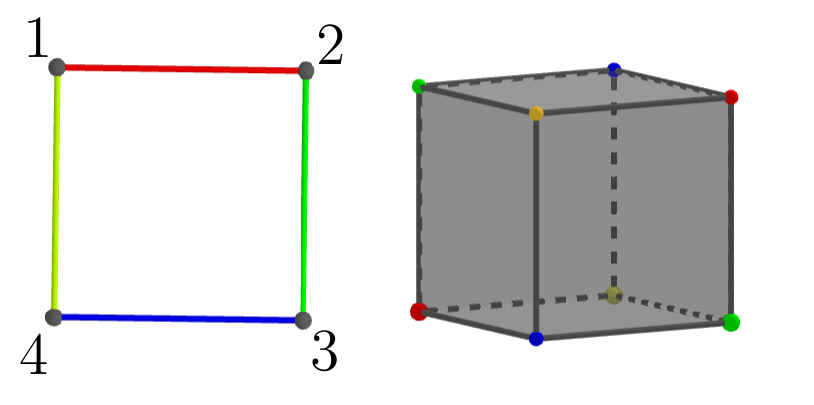}

\captionsetup{justification=centering}
\caption{Some examples of symmetric edge polytopes.\\The pictures were obtained using GeoGebra \cite{geogebra}.}
\end{figure}

Equivalently, after assigning an arbitrary orientation to each of the edges of $G$, one has that \[\Pc_G = \conv[M_G \ \! | \! \ -M_G],\]
where $M_G \in \ZZ^{|V|\times|E|}$ is the signed incidence matrix of $G$ with respect to the chosen orientation (i.e., the matrix whose $(v,e)$-entry is $1$ if $v$ is the head of $e$, $-1$ if $v$ is the tail of $e$, and $0$ otherwise). The matrix $M_G$ also serves as a representation of the \emph{graphic matroid} $\cM_G$ associated with $G$. Several objects associated with $\Pc_G$, for instance the facets or some triangulations, can be described via the combinatorial features of the graph $G$, and one can rephrase many of these characterizations in terms of the matroid $\cM_G$ only. This is not by accident; in fact, we will prove in \Cref{thm:SEPs via graphic matroids} that two symmetric edge polytopes $\Pc_G$ and $\Pc_H$ are unimodularly equivalent precisely when the graphical matroids $\cM_G$ and $\cM_H$ are isomorphic. In particular, if $G$ and $H$ are both $3$-connected, applying Whitney's 2-isomorphism theorem yields that $\Pc_G$ and $\Pc_H$ are unimodularly equivalent if and only if $G$ and $H$ are isomorphic. We remark that the characterization in \Cref{thm:SEPs via graphic matroids} corrects an erroneous statement of Matsui, Higashitani, Nagazawa, Ohsugi and Hibi \cite[Lemma 4.4, Theorem 4.5]{FirstSEP}: see \Cref{rem:unimodularly equivalent SEPs} for the details.

It is tempting to ask what happens if we take the polytope defined by the convex hull of the columns of $\Mpm$ for a more general matrix $M$, and whether this object bears any relation to the matroid represented by $M$. The former question was investigated by Ohsugi and Hibi in \cite{OhsugiHibiCS}, while the latter appears to be new. In general, changing the representation of a given matroid and applying the above ``symmetrization'' will produce wildly different polytopes, so this question might seem to be too far-fetched at first sight. However, as we show in \Cref{thm:well defined}, the construction described above does indeed yield a unique lattice polytope (up to some unimodular equivalence not involving any translation) when we consider any \emph{regular} matroid $\cM$ and restrict to its (full-rank) \emph{weakly unimodular} representations. We will call any polytope arising from such a setting a \emph{generalized symmetric edge polytope}. Throughout the paper, if we need the concrete polytope associated with a specific representation $M$, we will denote it by $\Pc_M$; if instead it is enough for our purposes to just deal with the equivalence class, we will write $\Pc_{\cM}$.

Most of the properties that make the usual symmetric edge polytopes pleasant are preserved in this wider environment: for instance, generalized symmetric edge polytopes are reflexive (as already observed in \cite{OhsugiHibiCS}) and terminal, and it is possible to describe their facets in a purely combinatorial fashion (\Cref{thm:facets as spanning 2-cuts}). We remark here that K\'alm\'an and T\'othm\'er\'esz have been working on a similar statement for extended root polytopes in the recent preprint \cite{KT_interior}.

The polars of generalized symmetric edge polytopes are special instances of \emph{Lipschitz polytopes} and enjoy a rich Ehrhart theory. More precisely, in the spirit of work by Beck and Zaslavsky \cite{BeZa}, we show that the lattice points in the $k$-th dilation of the polar $\Pc^{\Delta}_{\cM}$ are in bijection with the $(k+1)$-cuts of $\cM$ or, equivalently, with the $(k+1)$-flows of the dual matroid $\cM^*$ (\Cref{prop:polar lattice points}).

Finally, the existence of a regular unimodular triangulation for $\Pc_M$ had already been proved by Ohsugi and Hibi \cite{OhsugiHibiCS}, while an explicit one had been provided in the case of graphs by Higashitani, Jochemko and Micha{\l}ek \cite{HJM}. We show that, via a careful analysis of signed circuits, it is possible to extend the latter result to generalized symmetric edge polytopes (\Cref{thm:gb}).

The paper is organized as follows. \Cref{sec:preliminaries} contains some preliminaries about matroids, polytopes and toric ideals, while \Cref{sec:uniqueness} is devoted to define generalized symmetric edge polytopes and prove that any two full-rank weakly unimodular representations of the same regular matroid will yield unimodularly equivalent polytopes (\Cref{thm:well defined}).
\Cref{sec:properties} studies properties of generalized symmetric edge polytopes, including a partial converse to \Cref{thm:well defined}. \Cref{sec:polar} focuses on the polytopes polar to generalized symmetric edge polytopes and their Ehrhart theory; the obtained results are then used to derive a facet description for generalized symmetric edge polytopes, extending the one for the graphical case from \cite{HJM}. 
\Cref{sec:gb} is devoted to the explicit description of a regular unimodular triangulation of any generalized symmetric edge polytope.
Finally, we collect some open questions and suggestions for future work in \Cref{sec:future}.

\section{Preliminaries} \label{sec:preliminaries}

\subsection{Regular matroids, cuts and flows}
The aim of this subsection is to briefly introduce regular matroids and their properties. We direct the reader to \cite{Oxley} for a more complete treatment and for general matroid terminology.

\begin{notation}
If $\cM$ is a matroid, we will denote by $\Bk(\cM)$ and $\Ck(\cM)$ the sets of its bases and circuits, respectively. If $M$ is a matrix, we will sometimes write ``basis/circuit of $M$'' to refer to a basis/circuit of the matroid represented by $M$. In this case, the ground set of such a matroid will consist of the column indices of $M$.
\end{notation}

\begin{definition}
    Let $M \in \ZZ^{m \times n}$ be an integer matrix. We will say that $M$ is:
    \begin{itemize}
        \item \emph{totally unimodular} if the determinant of every square submatrix of $M$ lies in $\{0, \pm 1\}$;
        \item \emph{weakly unimodular} if the determinant of every square submatrix of $M$ of size $\max\{m,n\}$ lies in $\{0, \pm 1\}$. 
    \end{itemize}
\end{definition}

\begin{defthm} \label{defthm:regular}
A matroid $\cM$ of rank $r > 0$ is called \emph{regular} if it satisfies any of the following equivalent properties:
\begin{enumerate}
    \item $\cM$ can be represented via a totally unimodular matrix;
    \item $\cM$ can be represented via a full-rank weakly unimodular matrix;
    \item $\cM$ is representable over any field.
\end{enumerate}
\end{defthm}
\begin{proof}
The equivalence of (i) and (iii) is a well-known fact, see for instance \cite[Theorem 6.6.3]{Oxley}. We will now prove for clarity's sake that (i) and (ii) are equivalent: for a source in the literature, the reader can check \cite[Theorem 3.1.1]{White}.

To see that (i) implies (ii) it is enough to show that, if $\cM$ is represented by a totally unimodular matrix, then it is also represented by a totally unimodular (and hence weakly unimodular) matrix of the form $[I_r \ \! | \ \! D]$, where $r$ is the rank of $\cM$. For a proof of this claim, see \cite[Lemma 2.2.21]{Oxley}.

Let us now prove that (ii) implies (i). By assumption, there exists a full-rank matrix $M \in \ZZ^{r \times n}$ that is weakly unimodular and represents $\cM$. We now proceed as in \cite[Subsection 2.2]{SuWa}: after choosing a basis $\Bc$ for $\cM$, we can shuffle the columns of $M$ so that the elements of $\Bc$ correspond to the first $r$ columns. This amounts to multiplying $M$ on the right by an $(n \times n)$-permutation matrix $P$, an operation preserving the weakly unimodular property. Now consider the invertible submatrix $N$ of $MP$ obtained by taking the first $r$ columns. Since $MP$ is weakly unimodular and $N$ is invertible, the determinant of the integer matrix $N$ is either $1$ or $-1$; in other words, $N \in \mathrm{GL}_r(\ZZ)$. By construction, one has that $N^{-1}MP = [I_r \ \! | \! \ D]$ represents $\cM$ and is weakly unimodular; however, since it contains the identity matrix $I_r$, it must actually be totally unimodular (\cite[Lemma 3]{SuWa} or \cite[Exercise 10.1.1]{Oxley}), as desired.
\end{proof}

We illustrate the content of the previous definition with an example that will also serve as a running example throughout.

\begin{example} \label{ex:running}
    Let $\cM$ be the rank $3$ simple matroid with ground set $[5]$, bases $\Bk(\cM) = \{123, 124, 134, 135, 145, 234, 235, 245\}$, and circuits $\Ck(\cM) = \{125, 345, 1234\}$, where we are using the shorthand $i_1i_2\ldots i_m$ for $\{i_1, i_2, \ldots, i_m\}$. It is easy to check that $\cM$ is represented by the full-rank totally unimodular matrix \[M = \begin{bmatrix}1 & 0 & 0 & -1 & 1\\0 & 1 & 0 & -1 & 1\\0 & 0 & 1 & -1 & 0\end{bmatrix},\]
    and thus $\cM$ is regular. In fact, in this case $\cM$ is also graphic.
\end{example}

\begin{remark}
The assumption about the rank of $\cM$ being nonzero is not part of the usual definition of regular matroid in the literature: we include it to avoid nuisances with representability, see for instance \cite[Lemma 2.2.21]{Oxley}.
\end{remark}

The class of regular matroids (including those of rank zero) is closed under duality and contains all graphic matroids.

We now introduce cuts and flows of a regular matroid, following Su and Wagner's treatment in \cite{SuWa}.
Let $M \in \RR^{r \times n}$ (where $0 < r \leq n$) be a full-rank weakly unimodular matrix. We define the \emph{lattice of integer cuts} of $M$, denoted $\Gamma(M)$, and the \emph{lattice of integer flows} of $M$, denoted $\Lambda(M)$, as 
\begin{align*}
\Gamma(M) &\coloneqq \mathrm{row}(M) \cap \ZZ^n,\\\Lambda(M) &\coloneqq \ker(M) \cap \ZZ^n.
\end{align*}

The lattices of integer cuts and flows are orthogonal to each other with respect to the usual dot product. In particular, if $A = [I_r \ | \ D] \in \RR^{r \times n}$ (with $0 < r < n$) is totally unimodular and $A^* \coloneqq [-D^T \ | \ I_{n-r}]$, one has that
\begin{equation} \label{eq:flow-cut duality}
    \Lambda(A) = \Gamma(A^*).
\end{equation}

\begin{definition} \label{def:signed circuits}
    If $\vf \in \RR^n$, the \emph{support} of $\vf$, denoted by $\supp(\vf)$, is the set of indices $i \in [n]$ such that $v_i \neq 0$.
    
    Given a full-rank weakly unimodular matrix $M \in \RR^{r \times n}$ with $r \leq n$, we will call a flow $\blambda \in \Lambda(M)$ (respectively, a cut $\blambda \in \Gamma(M)$)
\begin{itemize}
    \item \emph{nowhere-zero} if $\lambda_i \neq 0$ for every $i \in [n]$, i.e., if $\supp(\blambda) = [n]$;
    \item a \emph{$k$-flow} (respectively, a \emph{$k$-cut}) if $|\lambda_i| < k$ for every $i \in [n]$;
    \item a \emph{signed circuit} or \emph{simple flow} if it is a $2$-flow and its support is a circuit of $M$. We denote the set of signed circuits of $M$ by $\sgncirc(M)$.
\end{itemize}
\end{definition}

For a regular matroid $\cM$, we are able to talk about \emph{the} lattice of integer cuts of $\cM$ \emph{up to isometry}: in fact, if $M$ and $M'$ are two full-rank weakly unimodular matrices representing $\cM$, then the elements of $\Gamma(M')$ correspond to elements of $\Gamma(M)$ via multiplication by a signed permutation matrix (compare, e.g.,~\cite[proof of Lemma 10]{SuWa}; their argument is stated for totally unimodular matrices, but goes through for full-rank weakly unimodular matrices as well). In particular, an element of $\Gamma(M')$ will be a nowhere-zero cut, a $k$-cut or a signed circuit if and only if the corresponding element of $\Gamma(M)$ is. Moreover, due to \eqref{eq:flow-cut duality}, all the above statements go through for flows as well.

\begin{example} \label{ex:cuts and flows}
The matrix $M$ from \Cref{ex:running} has
\begin{itemize}
    \item seventeen $2$-cuts: $(0,0,0,0,0)$, $(1,0,0,-1,1)$, $(-1,0,0,1,-1)$, $(0,1,0,-1,1)$,\\
    $(0,-1,0,1,-1)$, $(0,0,1,-1,0)$, $(0,0,-1,1,0)$, $(1,-1,0,0,0)$, $(-1,1,0,0,0)$,\\ $(1,0,-1,0,1)$, $(-1,0,1,0,-1)$, $(0,1,-1,0,1)$, $(0,-1,1,0,-1)$, $(1,-1,1,-1,0)$,\\ $(-1,1,-1,1,0)$, $(1,-1,-1,1,0)$, $(-1,1,1,-1,0)$.
    \item seven $2$-flows: $(0,0,0,0,0)$, $(1,1,0,0,-1)$, $(-1,-1,0,0,1)$, $(0,0,1,1,1)$,\\ $(0,0,-1,-1,-1)$, $(1,1,1,1,0)$, $(-1,-1,-1,-1,0)$.
    \item six signed circuits: all the $2$-flows except for the origin.
\end{itemize}
\end{example}

We record here for further reference some useful facts:

\begin{proposition}
\label{prop:useful facts on signed circuits}
Let $M$ be a full-rank weakly unimodular matrix.
\begin{enumerate}
    \item If $\blambda \in \Lambda(M)$ and $\supp(\blambda)$ is a circuit of $M$, then every coordinate of $\blambda$ has the same absolute value.
    \item If $\Cc \in \Ck(M)$, then there are exactly two signed circuits (differing by a global sign) with support $\Cc$.
\end{enumerate}
\end{proposition}
\begin{proof}
Part (i) can be derived directly from \cite[Lemma 4.9]{Stu} and constitutes a strengthening of \cite[Lemma 7]{SuWa}. The proof of part (ii) is almost verbatim the same as the one of \cite[Lemma 8]{SuWa}, using part (i) instead of \cite[Lemma 7]{SuWa}.
\end{proof}

\begin{definition}
    Let $\cM$ be a regular matroid with ground set $E$. If $\Bc$ is a basis of $\cM$ and $e \in E \setminus \Bc$, then the \emph{fundamental circuit of $e$ with respect to $\Bc$} is the unique circuit $\Cc(e,\Bc) \in \Ck(\cM)$ contained in $\Bc \cup \{e\}$. Note that $e \in \Cc(e,\Bc)$.
    
    If, moreover, $M$ is a full-rank weakly unimodular representation of $\cM$, then by \Cref{prop:useful facts on signed circuits} there is a unique signed circuit $\overrightarrow{\Cc}(e,\Bc) \in \Lambda(M)$ supported at $\Cc(e,\Bc)$ whose $e$-th entry equals $1$. We will call such a signed circuit the \emph{fundamental signed circuit of $e$} with respect to $\Bc$ and $M$.
\end{definition}

\begin{example}
    Let $\cM$ and $M$ be as in \Cref{ex:running}. Then, for $\Bc = \{1,2,3\}$, one has that $\overrightarrow{\Cc}(4,\Bc) = (1,1,1,1,0)$ and $\overrightarrow{\Cc}(5,\Bc) = (-1,-1,0,0,1)$.
\end{example}

\begin{lemma} \label{lem:unique cut}
Let $M \in \RR^{r \times n}$ (where $0 < r \leq n$) be a full-rank weakly unimodular matrix and assume that the first $r$ columns of $M$ are linearly independent. Then, for any $a_1, \ldots, a_r \in \ZZ$, there exist unique $a_{r+1}, \ldots, a_n \in \ZZ$ such that $(a_1, \ldots, a_n) \in \Gamma(M)$. In other words, there exists a unique cut $\boldsymbol{\gamma} \in \Gamma(M)$ having $a_1, \ldots, a_r$ as its first $r$ entries.
\end{lemma}
\begin{proof}
Call $\cM$ the regular matroid represented by $M$.

If $r=n$, then $\Gamma(M) = \ZZ^n$ and the claim is true. Assume now that $r < n$.

To prove existence, define $\boldsymbol{\gamma} \in \ZZ^n$ in the following way:
\begin{itemize}
    \item $\gamma_i = a_i$ for every $i \in [r]$;
    \item for every $j \in \{r+1, \ldots, n\}$, we determine $\gamma_j$ by imposing that $\boldsymbol{\gamma} \cdot \overrightarrow{\Cc}(j, [r]) = 0$, where $\overrightarrow{\Cc}(j, [r])$ is the fundamental signed circuit of $j$ with respect to the basis $[r]$ of $\cM$ and the representation $M$.
    \end{itemize}
    To prove that the integer vector $\boldsymbol{\gamma} \in \ZZ^n$ we have just defined is indeed a cut, it is enough to show that $\boldsymbol{\gamma} \in \mathrm{row}(M)$; but since $\mathrm{row}(M)$ and $\ker(M)$ are orthogonal with respect to the standard dot product, this amounts to proving that $\boldsymbol{\gamma} \cdot \vf = 0$ for every $\vf \in \ker(M)$. Since the fundamental signed circuits of $M$ with respect to $[r]$ form an $\RR$-basis of $\ker(M)$ (being $n-r$ many linearly independent vectors by construction), the claim follows.

To prove uniqueness, assume there is another cut $\boldsymbol{\gamma'}$ such that $\gamma'_i = a_i$ for every $i \in [r]$, and consider $\boldsymbol{\beta} \coloneqq \boldsymbol{\gamma'} - \boldsymbol{\gamma} \in \Gamma(M)$. By assumption, $\beta_i = 0$ for every $i \in [r].$ Since $\mathrm{row}(M) = \ker(M)^{\perp}$, it follows that $\boldsymbol{\beta} \cdot \boldsymbol{\lambda} = 0$ for every $\boldsymbol{\lambda} \in \ker(M)$. In particular, for every $j \in \{r+1, \ldots, n\}$, one has that $\boldsymbol{\beta} \cdot \overrightarrow{\Cc}(j, [r]) = 0$, and thus $\beta_j = 0$. This proves that $\boldsymbol{\gamma'} = \boldsymbol{\gamma}$.
\end{proof}

Given a matroid $\cM$, denote by $\cM^{\circ}$ the matroid obtained from $\cM$ by deleting all its loops. Su and Wagner proved in \cite{SuWa} that knowing the lattice of integer cuts of $\cM$ is enough to determine $\cM^{\circ}$ up to isomorphism. In particular, if we know beforehand that $\cM$ is loopless (for instance, if $\cM$ is simple), we can reconstruct completely $\cM$ from the data of its lattice of integer cuts. This idea will serve as a blueprint for the constructions in this paper.

\subsection{Polarity}
Let $P \subseteq \RR^d$ be a full-dimensional lattice polytope with $\origin \in \mathring{P} \cap \ZZ^d$ (here $\mathring{P}$ denotes the interior of $P$ with respect to the Euclidean topology). 
We recall that the \emph{polar} of $P$ is the polytope 
\[P^{\Delta} \coloneqq \{\uf \in \RR^d \mid \uf \cdot \xf \leq 1 \text{ for every }\xf \in P\},\]

where we are using the usual dot product to identify $\RR^d$ and its dual $(\RR^d)^*$.

The polar $P^{\Delta}$ will not be a lattice polytope in general. If $P^{\Delta}$ happens to be a lattice polytope, then $P$ is called \emph{reflexive}.

For the rest of this subsection we fix a full-dimensional reflexive polytope $P \subseteq \RR^d$ with $\mathring{P} \cap \ZZ^d = \{\origin\}$.

\begin{notationremark} \label{not:F_u}
If $\uf \in P^{\Delta}\cap \ZZ^d$, we denote by $F_{\uf}$ the face of $P$ obtained as $\conv\{\vf_i \mid \vf_i \cdot \uf = 1\}$, where the $\vf_i$'s are the vertices of $P$. 
Indeed, the polytope $P$ lies entirely inside one of the halfspaces defined by the hyperplane $H_{\uf} \coloneqq \{\xf \in \RR^d \mid \uf \cdot \xf = 1\}$.

    By polarity, facets of the polytope $P$ correspond to the vertices of the polar polytope $P^{\Delta}$; in particular, any facet of $P$ will be of the form $F_{\uf}$ for some $\uf \in P^{\Delta} \cap \ZZ^d$, and such a $\uf$ will be a vertex of $P^{\Delta}$. 
\end{notationremark}

\subsection{Toric ideals}
We introduce some basic notation about toric ideals. For the concepts not explained here and to get further insight, see for instance \cite[Chapters 4 and 8]{Stu}.

\begin{definition} \label{def:toric circuits}
If $M \in \ZZ^{r \times n}$ is an integer matrix with columns $\vf_1, \ldots, \vf_n$, we will denote by $I_M$ the \emph{toric ideal} associated with $M$, i.e.~the kernel of the map 
\[
\begin{aligned}
\pi\colon\  K[x_1, \ldots, x_n] &\to K[t_1^{\pm 1}, \ldots, t_r^{\pm 1}]\\
x_i &\mapsto \mathbf{t}^{\vf_i} \coloneqq t_1^{m_{1,i}}t_2^{m_{2,i}}\ldots t_r^{m_{r,i}},
\end{aligned}
\]
where $K$ is a field. Every $\blambda \in \ZZ^n$ can be uniquely written as $\blambda^+ - \blambda^-$, where $\blambda^+$ and $\blambda^-$ are in $\NN^n$ and have disjoint supports. Any column vector $\blambda \in \ker(M) \cap \ZZ^n$ gives rise to a binomial $\xf^{\blambda^+} - \xf^{\blambda^-} \in \ker(\pi)$, and the ideal $I_M$ is generated by binomials of this form. In what follows, with a slight abuse of notation, we will use the expression ``signed circuit'' to denote both an element $\boldsymbol{\lambda} \in \{0, \pm 1\}^n$ as in \Cref{def:signed circuits} and the associated binomial $\xf^{\blambda^+} - \xf^{\blambda^-}$ in $\ker(\pi)$.
\end{definition}

\begin{remark} \label{rem:unimodular toric}
When $M$ is full-rank weakly unimodular, the toric ideal $I_M$ is remarkably well-behaved: in fact, the set $\sgncirc(M)$ of signed circuits is a universal Gr\"obner basis for $I_M$ (and hence, in particular, the signed circuits of $M$ generate $I_M$).
Actually, an even stronger result is true, as $\sgncirc(M)$ turns out to be the \emph{Graver basis} of $I_M$ \cite[Theorem 8.11]{Stu}. (In fact, since two signed circuits only differing by a global sign give rise to the same binomial up to sign, one usually picks a representative for every pair; in particular, the Graver basis of $I_M$ will have cardinality $|\Ck(M)| = \frac{1}{2}|\sgncirc(M)|$.)

\end{remark}

\begin{example}
    Let $M$ be as in \Cref{ex:running}. The enumeration of signed circuits in \Cref{ex:cuts and flows} shows that the polynomials $x_1x_2-x_5$, $x_3x_4x_5-1$ and $x_1x_2x_3x_4 - 1$ are the Graver basis (and a universal Gr\"obner basis) of the toric ideal $I_M$.
\end{example}

Throughout the paper, when we say that a certain polynomial inside a polynomial ring is homogeneous, we are using the standard grading: i.e.,~each variable has degree $1$. We record here for further reference a useful observation.

\begin{lemma} \label{lem:toric homogenization}
Let $B \in \ZZ^{m \times n}$ and let $B' \in \ZZ^{(m+1) \times (n+1)}$ be the matrix defined via \[b'_{ij} \coloneqq \begin{cases}b_{ij} & \text{if }i\leq m \text{ and }j \leq n \\0 & \text{if }i\leq m \text{ and }j = n+1 \\ 1 & \text{if }i=m+1 \end{cases}.\] Let $I_B \subseteq K[x_1, \ldots, x_n]$ and $I_{B'} \subseteq K[x_1, \ldots, x_n, z]$ be the respective toric ideals (here $x_i$ corresponds to the $i$-th column and $z$ to the $(n+1)$-st, when available). Then $I_{B'} = I_B^{\mathrm{hom}}$, where the homogenization is taken with respect to the variable $z$.
\end{lemma}

\begin{proof}
Let us first prove that $I_{B'} \subseteq I_B^{\mathrm{hom}}$. The toric ideal $I_{B'}$ is generated by the set of its primitive binomials, i.e.~its Graver basis. Let $f$ be a primitive binomial of $I_{B'}$. Due to primitivity, the variable $z$ can appear at most on one side of the binomial; without of loss of generality, we can hence write $f = \xf^{\blambda^+} - \xf^{\blambda^-}z^k$, where $\blambda = \blambda^+ - \blambda^- \in \ZZ^n$, $k \geq 0$ and $(\blambda, k) \in \ker(B')$. By construction, $\blambda \in \ker(B)$ and $f$ is homogeneous; more precisely, $f$ is the homogenization of a binomial in $I_B$ with respect to the variable $z$. It follows that $I_{B'} \subseteq I_B^{\mathrm{hom}}$.

Let us now prove that $I_B^{\mathrm{hom}} \subseteq I_{B'}$. By \cite[Theorem 8.4.4]{CLO}, in order to find a generating set for $I_B^{\mathrm{hom}}$ it is enough to homogenize a set of polynomials forming a Gr\"obner basis of $I_B$ with respect to a graded monomial order. Primitive polynomials provide such a set: in fact, the Graver basis of $I_B$ contains the universal Gr\"obner basis of $I_B$ \cite[Lemma 4.6]{Stu}. Let $g = \xf^{\blambda^+} - \xf^{\blambda^-}$ be a primitive binomial in $I_B$. We can assume without loss of generality that $k \coloneqq |\blambda^+| - |\blambda^-| \geq 0$. By construction, the homogenized polynomial $\xf^{\blambda^+} - \xf^{\blambda^-}z^k$ lies in $I_{B'}$. This shows that $I_B^{\mathrm{hom}} \subseteq I_{B'}$.
\end{proof}

\section{Uniqueness up to unimodular equivalence} \label{sec:uniqueness}
The main aim of this section is to describe how to extend the definition of a symmetric edge polytope from the context of graphs to that of regular matroids. Let us first fix some notation:

\begin{notation} \label{not:PM and psiF}
For any integer matrix $M \in \ZZ^{r \times n}$ with $0 < r \leq n$, we denote by $\Pc_M$ the lattice polytope of $\RR^r$ obtained as $\conv\begin{bmatrix}M \ | \ -M\end{bmatrix}$.

For $F \in \mathrm{GL}_r(\ZZ)$, we denote by $\psi_F: \mathbb{R}^r \to \mathbb{R}^r$ the affine map sending $\xf$ to $F\xf$.
\end{notation}

 It is our goal to show that any two full-rank weakly unimodular representations of a regular matroid $\cM$ produce the same lattice polytope (in the sense specified in \Cref{not:PM and psiF}) up to unimodular equivalence, and the same holds for the polytopes obtained via polarity. More precisely, we show the following:

\begin{theorem} \label{thm:well defined}
Let $\cM$ be a regular matroid of rank $r > 0$ on $n$ elements and let $M_1, M_2 \in \mathbb{R}^{r \times n}$ be two full-rank weakly unimodular representations of $\cM$. Then there exists $F\in \mathrm{GL}_r(\ZZ)$ such that \[\Pc_{M_2} = \psi_F(\Pc_{M_1}) \quad \text{and} \quad \Pc_{M_2}^{\Delta} = \psi_{(F^T)^{-1}}(\Pc_{M_1}^{\Delta}).\]
\end{theorem}

We will show how to handle the case when the matrices representing $\cM$ are not full-rank in \Cref{rem:not full rank}. Moreover, a partial converse to \Cref{thm:well defined} will be proved later, see \Cref{thm:same simple matroid}.

\begin{proof}
Pick two weakly unimodular full-rank $(r \times n)$-matrices $M_1$ and $M_2$ both representing $\cM$. For each $i \in \{1,2\}$, write $\Pc_i \coloneqq \Pc_{M_i}$. Multiplying $M_i$ on the right by a (signed) permutation matrix has no effect on the polytope $\Pc_i$: permuting the columns just permutes the list $L$ of points we are taking the convex hull of, and changing the sign of a column is harmless because the list $L$ consists of the columns of both $M_i$ and $-M_i$. After some permutation of the columns of $M_1$ and $M_2$, we can hence assume without loss of generality the following two statements:
\begin{itemize}
    \item the identity map $[n] \to [n]$ yields an isomorphism between the matroids represented by $M_1$ and $M_2$;
    \item the submatrices $N_1$ and $N_2$ obtained by selecting the first $r$ columns of respectively $M_1$ and $M_2$ are both invertible.
\end{itemize}

Proceeding as in the proof of ``(ii) implies (i)'' in \Cref{defthm:regular}, we can now multiply each $M_i$ on the left by $N_i^{-1} \in \mathrm{GL}_r(\ZZ)$, obtaining the totally unimodular matrix $[I_r \ \! | \! \ D_i]$. Since the identity map still yields an isomorphism between the matroids represented by $[I_r \ \! | \! \ D_1]$ and $[I_r \ \! | \! \ D_2]$, we can apply \cite[Proposition 6.4.1]{Oxley} to get that $D_1$ and $D_2$ are congruent modulo 2, and hence so are $[I_r \ \! | \! \ D_1]$ and $[I_r \ \! | \! \ D_2]$. Since $[I_r \ \! | \! \ D_1]$ and $[I_r \ \! | \! \ D_2]$ are both totally unimodular, we are now in the position to use Camion's signing lemma \cite[Lemma 10.1.7]{Oxley}: i.e., we can obtain the matrix $[I_r \ \! | \! \ D_2]$ by changing the signs of some rows and columns of $[I_r \ \! | \! \ D_1]$. In other words, there exist diagonal matrices $R \in \mathrm{GL}_r(\ZZ)$ and $C \in \mathrm{GL}_n(\ZZ)$ with only $1$'s and $-1$'s on the diagonal and such that $[I_r \ \! | \! \ D_2] = R \cdot [I_r \ \! | \! \ D_1] \cdot C$.

Now let $F \coloneqq N_2 \cdot R \cdot N_1^{-1} \in \mathrm{GL}_r(\ZZ)$. It follows from the discussion above that $\Pc_2 = \psi_F(\Pc_1)$, as desired (note that $C$, being a signed permutation matrix, does not enter the picture).

The polar statement can now be derived like this:
\[
\begin{split}
    \Pc_2^{\Delta} &= \{\uf \in \RR^r \mid \uf \cdot \xf \leq 1 \text{ for every }\xf \in \Pc_2\}\\
    &= \{\uf \in \RR^r \mid \uf \cdot \xf \leq 1 \text{ for every }\xf \in \psi_F(\Pc_1)\}\\
    &= \{\uf \in \RR^r \mid \uf \cdot F\yf \leq 1 \text{ for every }\yf \in \Pc_1\}\\
    &= \{\uf \in \RR^r \mid F^T\uf \cdot \yf \leq 1 \text{ for every }\yf \in \Pc_1\}\\
    &= \{(F^T)^{-1}\vf \in \RR^r \mid \vf \cdot \yf \leq 1 \text{ for every }\yf \in \Pc_1\}\\
    &= \psi_{(F^T)^{-1}}(\Pc_1^{\Delta}).
\end{split}
\]
\end{proof}

\begin{example} \label{ex:unimodular equivalence}
Let $\cM$ be the uniform matroid $U_{2,3}$. The two full-rank totally unimodular matrices \[M_1 \coloneqq \begin{bmatrix}1 & 0 & 1\\ 0 & 1 & 1 \end{bmatrix} \text{ and } M_2 \coloneqq \begin{bmatrix}1 & 0 & 1\\ 0 & 1 & -1 \end{bmatrix}\] both represent $\cM$. Changing the signs of both the second row and the second column of $M_1$ yields $M_2$; in formulas,
\[\begin{bmatrix}1 & 0 & 1\\ 0 & 1 & -1\end{bmatrix} = \begin{bmatrix}1 & 0\\ 0 & -1 \end{bmatrix}\begin{bmatrix}1 & 0 & 1\\ 0 & 1 & 1\end{bmatrix}\begin{bmatrix}1 & 0 & 0\\ 0 & -1 & 0 \\ 0 & 0 & 1\end{bmatrix}.\]
It is easy to verify that the two polytopes $\Pc_1$ and $\Pc_2$ are unimodularly equivalent, as guaranteed by \Cref{thm:well defined}.
\end{example}

The usual symmetric edge polytope associated with a graph $G$ is defined as $\Pc_{A_G}$, where $A_G$ is any signed incidence matrix associated with $G$. The matrix $A_G$ provides a totally unimodular representation of the graphic matroid $\cM_G$, but is \emph{not} full-rank. However, this is not really an issue, as we now explain.

\begin{remark} \label{rem:not full rank}
Let $\cM$ be a regular matroid of rank $r > 0$ and let $M \in \ZZ^{m \times n}$ be a totally unimodular representation of $\cM$ with $m > r$. Possibly after permuting the columns of $M$, we can assume without loss of generality that the first $r$ columns of $M$ are linearly independent. Pivoting repeatedly we can then reach a matrix \[M' \coloneqq \begin{bmatrix}I_r & D \\ \mathbf{0}_{m-r, r} & \mathbf{0}_{m-r, n-r}\end{bmatrix},\] which will again be totally unimodular by \cite[Lemma 2.2.20]{Oxley}.

The two polytopes $\Pc_M$ and $\Pc_{M'}$ are unimodularly equivalent, and projecting onto the first $r$ coordinates shows that $\Pc_{M'}$ is in turn unimodularly equivalent to $\Pc_{M''}$, where $M'' \coloneqq \begin{bmatrix}I_r \ | \ D\end{bmatrix}$ is a full-rank totally unimodular representation of $\cM$.
\end{remark}

\begin{example}\label{ex:3-cycle}
Let $G = C_3$ be the cycle graph on three vertices and pick \[A_G \coloneqq \begin{bmatrix}1 & 1 & 0 \\ -1 & 0 & -1 \\ 0 & -1 & 1\end{bmatrix}.\]
Then $\Pc_{A_G}$ is the symmetric edge polytope $\Pc_G$. Successive row operations on $A_G$ yield that
\[\begin{bmatrix}1 & -1 & 0 \\ 0 & 1 & 0 \\ 0 & 1 & 1\end{bmatrix} \begin{bmatrix}1 & 0 & 0 \\ 1 & 1 & 0 \\ 0 & 0 & 1\end{bmatrix} \begin{bmatrix}1 & 1 & 0\\ -1 & 0 & -1 \\ 0 & -1 & 1\end{bmatrix} = \begin{bmatrix}1 & 0 & 1 \\ 0 & 1 & -1 \\ 0 & 0 & 0\end{bmatrix},\]
and so $\Pc_G$ is unimodularly equivalent to the full-dimensional polytope $\Pc_{M_2} \subseteq \RR^2$, with $M_2$ as in \Cref{ex:unimodular equivalence}.

\begin{remark} \label{rem:TU full-rank for G}
Selecting a directed spanning tree inside a connected finite simple graph $G$ on $r+1$ vertices and $n$ edges yields an explicit full-rank totally unimodular representation for $\mathcal{M}_G$ in the following way (compare \cite[Section 5.1, p.~137]{Oxley}):
\begin{itemize}
    \item fix an orientation for each edge of $G$;
    \item pick a spanning tree $\mathcal{T}$ for $G$ and number its edges from $1$ to $r$;
    \item assign the $i$-th standard basis vector $\ef_i \in \mathbb{R}^r$ to the $i$-th edge in $\mathcal{T}$ (taken with the orientation selected at the beginning);
    \item for any edge $\vec{e}$ in $G$ taken with its orientation, consider the unique directed path $\mathcal{P}_{\vec{e}}$ from the starting vertex to the ending vertex that only uses edges of $\mathcal{T}$;
    \item assign to $\mathcal{P}_{\vec{e}}$ the vector $\vf_{\vec{e}} = (\lambda_1, \ldots, \lambda_r)$, where $\lambda_i$ equals $1$ if the $i$-th edge of $\mathcal{T}$ appears in $\mathcal{P}_{\vec{e}}$ with its ``correct'' orientation, $-1$ if it is traversed backwards, $0$ if it does not appear at all.
\end{itemize}
Putting together all the vectors $\vf_{\vec{e}}$ as columns of a matrix yields a full-rank totally unimodular matrix $[I_r \ \! \mid \ \! D]$ representing $\mathcal{M}_G$. By the results in this section, this also produces a full-dimensional polytope $\Pc_{[I_r \ \! \mid \ \! D]} \subseteq \RR^r$ unimodularly equivalent to the symmetric edge polytope of $G$ (compare this to \Cref{ex:3-cycle}). If the graph $G$ is not connected, one can select a directed spanning tree for each connected component and argue analogously.
\end{remark}
\end{example}

\section{First properties of generalized symmetric edge polytopes} \label{sec:properties}

Due to \Cref{thm:well defined}, if we are given a regular matroid $\cM$ of rank $r>0$ on $n$ elements, we know how to define a full-dimensional polytope $\Pc_{\cM} \subseteq \RR^r$ which is defined up to some unimodular equivalence not involving any translation. We now wish to prove some results about $\Pc_{\cM}$: in the proofs we will often need to fix a specific full-rank totally unimodular representation of $\cM$.

We begin by noting that the polytope $\Pc_{\cM}$ does not see potential loops or parallel elements inside the matroid $\cM$, in analogy to the usual symmetric edge polytopes (see \cite[Remark 60]{DDM}).

\begin{remark} \label{rem:simplification}
Let $\cM$ be a regular matroid of rank $r>0$ and let $M$ be a full-rank weakly unimodular matrix representing $\cM$. Let $\overline{M}$ be the submatrix of $M$ obtained by keeping only the nonzero columns $\cf_i$ such that $\cf_i \neq \pm\cf_j$ for every $j<i$. Then $\Pc_{\overline{M}} = \Pc_M$, since the redundant columns in $M$ do not affect the structure of $\Pc_M$ and, as $\origin$ always lies in the interior of $\Pc_M$, the same holds for the zero columns. 
Hence, the polytope $\Pc_{\cM}$ does not see loops or parallel elements of $\cM$; as a consequence, we can replace $\cM$ by its \emph{simplification}\footnote{In \cite[Definition 4.6]{AHK} this is called the \emph{combinatorial geometry} of $\cM$.} $\overline{\cM}$.
\end{remark}

\begin{notation}In the setting of \Cref{rem:simplification}, we will say $M$ has \emph{irredundant columns} if $\overline{M} = M$, i.e., if the regular matroid represented by $M$ is simple. 
\end{notation}

We now wish to collect some properties of generalized symmetric edge polytopes. We point out that parts (i) to (iii) of \Cref{thm:first properties} below were essentially already known to Ohsugi and Hibi \cite[Lemma 2.11]{OhsugiHibiCS}.

\begin{theorem} \label{thm:first properties}
Let $\cM$ be a regular matroid of rank $r>0$. The following properties hold:
\begin{enumerate}
\item $\Pc_{\cM}$ is centrally symmetric;
    \item $\dim(\Pc_{\cM}) = \mathrm{rk}(\cM)$;
    \item $\Pc_{\cM}$ is reflexive;
    \item $\Pc_{\cM}$ is terminal, i.e., the only points of $\Pc_{\cM}$ with integer coordinates are its vertices and the origin;
    \item The vertices of $\Pc_{\cM}$ are twice as many as the atoms of the lattice of flats of $\cM$. In particular, if $\cM$ is simple, every antipodal pair of vertices of $\Pc_{\cM}$ corresponds to an element in the ground set of $\cM$.

\end{enumerate}
\end{theorem}
\begin{proof}
Part (i) is immediate by definition, no matter which representation $M$ we choose for $\cM$.

Due to \Cref{thm:well defined} we know that, if $M$ and $M'$ are two full-rank weakly unimodular matrices representing $\cM$, we can go from $\Pc_{M}$ to $\Pc_{M'}$ and from $\Pc_{M}^{\Delta}$ to $\Pc_{M'}^{\Delta}$ via unimodular maps that do not involve any translation. In particular, it is enough to prove statements (ii)--(v) for $\Pc_{M}$, where $M$ is a full-rank totally unimodular $r \times n$ matrix representing $\cM$.

Part (ii) is now immediate and, together with part (i), implies that the origin lies in the interior of $\Pc_M$; hence, the polar polytope $\Pc_M^{\Delta}$ is well-defined, and an $\mathcal{H}$-presentation for it is given by $\Mpm^T \mathbf{x} \leq \mathbf{1}$. Since $M$ is totally unimodular, so is $\begin{bmatrix}M \ | \  -M\end{bmatrix}^T$; the polar $\Pc_M^{\Delta}$ must then be a lattice polytope (see for instance \cite[Theorem 19.1]{Schrijver}), and hence $\Pc_M$ is reflexive. This proves part (iii).

As regards part (iv), pick a lattice point $\mathbf{x} = (x_1, \ldots, x_r)$ of $\Pc_M$ different from the origin. Then we can write $\mathbf{x} = \sum_i \lambda_i \mathbf{v}_i$, where $\lambda_i > 0$, $\sum_i \lambda_i = 1$, and the vertices $\mathbf{v}_i$ form a set of pairwise distinct nonzero columns of $\Mpm$. If $r=1$, the claim is obvious. Assume hence that $r > 1$. Since $\mathbf{x} \neq \origin$ and $\Pc_M$ is a centrally symmetric subset of the hypercube $[-1,1]^r$, we can assume without loss of generality that $x_1 = 1$. Then the first coordinate of every $\mathbf{v}_i$ must also equal $1$. If $\mathbf{x} = \mathbf{v}_1$, there is nothing to prove. Assume otherwise. Then there is a coordinate (without loss of generality, the second one) in which $\mathbf{x}$ and $\mathbf{v}_1$ differ. This can happen only if $x_2 = 0$ and $(\mathbf{v}_1)_2 \in \{1, -1\}$. But then there must exist some $j > 1$ such that $(\mathbf{v}_j)_2 = -(\mathbf{v}_1)_2$. As a consequence, the totally unimodular matrix $\Mpm$ contains the submatrix \[\begin{bmatrix}1 & 1 \\ (\mathbf{v}_1)_2 & -(\mathbf{v}_1)_2\end{bmatrix}\] with determinant $2$ or $-2$. This yields a contradiction.

Finally, it is enough to prove the statement of part (v) when $\cM$ is simple. When this is the case, then $M$ has irredundant columns; denote by $\cf_1, \ldots, \cf_{2n}$ the columns of $\Mpm$. Assume by contradiction that a column of $\Mpm$ (without loss of generality, the first one) can be expressed as a convex combination of the other ones; i.e., $\cf_1 = \sum_{j \in J}\lambda_j \cf_j$ for some $J \subseteq \{2, 3, \ldots, 2n\}$, $\lambda_j > 0$, $\sum_{j \in J}\lambda_j = 1$. Since $M$ has no zero columns and $\Pc_M$ is a centrally symmetric subset of the hypercube $[-1,1]^r$, we can assume without loss of generality that $(\cf_1)_1 = 1$, and this in turn implies that $(\cf_j)_1 = 1$ for every $j \in J$. Arguing in a similar way to part (iv), one can then build a submatrix of $\Mpm$ with determinant $2$ or $-2$, which in turn yields the desired contradiction.
\end{proof}
    \begin{figure}[h!] 
        \includegraphics[scale=0.5]{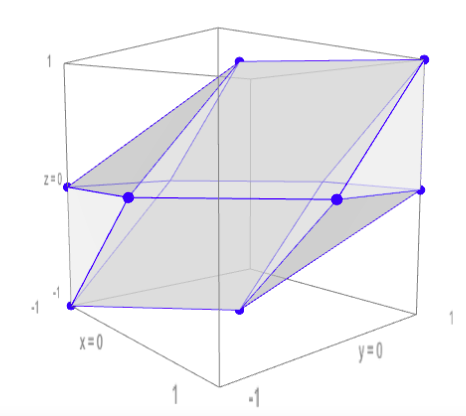}
        \caption{The polytope $\Pc_M$, where $M$ is the matrix from \Cref{ex:running}. Plot generated by SageMath \cite{sagemath}.}
        \label{fig:gSEP example}
    \end{figure}
    
\begin{example}
    Let $\cM$ and $M$ be as in \Cref{ex:running}. Then $\Pc_M$ is the polytope shown in \Cref{fig:gSEP example}. One has that $\dim \Pc_M = \mathrm{rk}(M) = 3$; since the matroid $\cM$ is simple, the lattice points of $\Pc_M$ are the origin and the columns of $\Mpm$.
\end{example}

\Cref{thm:first properties} gives us the tools to establish a partial converse to \Cref{thm:well defined}. 

\begin{theorem} \label{thm:same simple matroid}
Let $M, N \in \ZZ^{r \times n}$ (where $0 < r \leq n$) be two full-rank weakly unimodular matrices with irredundant columns, and assume that the polytopes $\Pc_M$ and $\Pc_N$ are unimodularly equivalent. Then there exist $F \in \mathrm{GL}_r(\ZZ)$ and a signed permutation matrix $P \in \ZZ^{n \times n}$ such that $N = FMP$. In particular, $N$ and $M$ represent the same simple regular matroid $\cM$.   
\end{theorem}
\begin{proof}
By assumption there exist $F \in \mathrm{GL}_r(\ZZ)$ and $\vf \in \ZZ^r$ such that $\Pc_{N} = \psi_F(\Pc_{M}) + \vf$. Since $\mathbf{0}$ is the only interior point of the reflexive polytopes $\Pc_{N}$ and $\Pc_{M}$, it must be that $\vf = \mathbf{0}$, so that no translation is actually involved. Moreover, one can easily check that $\Pc_N = \psi_F(\Pc_{M}) = \Pc_{FM}$.

Since the matrices $FM$ and $N$ are both full-rank and weakly unimodular, \Cref{thm:first properties}(v) implies that the columns of both $\FMpm$ and $\Npm$ correspond to the vertices of $\Pc_{FM} = \Pc_N$. As a consequence, the matrices $FM$ and $N$ can only differ by a signed permutation of their columns; in other words, there exists a signed permutation matrix $P \in \ZZ^{n \times n}$ such that $N = FMP$, as desired.
\end{proof}

As a consequence, we obtain that the matroidal setting is the ``right'' one to study even the usual symmetric edge polytopes.

\begin{theorem} \label{thm:SEPs via graphic matroids}
Let $G$ and $H$ be finite simple graphs. Then the symmetric edge polytopes $\Pc_G$ and $\Pc_H$ are unimodularly equivalent if and only if the graphic matroids $\cM_G$ and $\cM_H$ are isomorphic. 
\end{theorem}
\begin{proof}
The ``if'' part follows from \Cref{rem:TU full-rank for G} and \Cref{thm:well defined}. The ``only if'' part follows from \Cref{rem:TU full-rank for G} and \Cref{thm:same simple matroid}, noting that any signed incidence matrix of a simple graph has irredundant columns by construction.
\end{proof}

\begin{corollary}
    Let $G$ and $H$ be finite simple $3$-connected graphs. Then the symmetric edge polytopes $\Pc_G$ and $\Pc_H$ are unimodularly equivalent if and only if $G$ and $H$ are isomorphic.
\end{corollary}
\begin{proof}
    This follows directly from \Cref{thm:SEPs via graphic matroids} and Whitney's $2$-isomorphism theorem \cite[Theorem 5.3.1]{Oxley}.
\end{proof}

\begin{remark} \label{rem:unimodularly equivalent SEPs}
    It was claimed in \cite[Lemma 4.4]{FirstSEP} that, if $G$ and $H$ are finite simple graphs and $G$ is $2$-connected, then $\Pc_G$ and $\Pc_H$ are unimodularly equivalent if and only if $G$ and $H$ are isomorphic. Unfortunately, this claim is erroneous and affects the validity of \cite[Theorem 4.5]{FirstSEP} as well: indeed, there exist non-isomorphic $2$-connected graphs giving rise to the same graphic matroid, and thus having unimodularly equivalent symmetric edge polytopes by \Cref{thm:SEPs via graphic matroids}. The key to build such objects is the \emph{Whitney twist} operation, see \cite[Section 5.3]{Oxley}. We provide here an explicit example.

    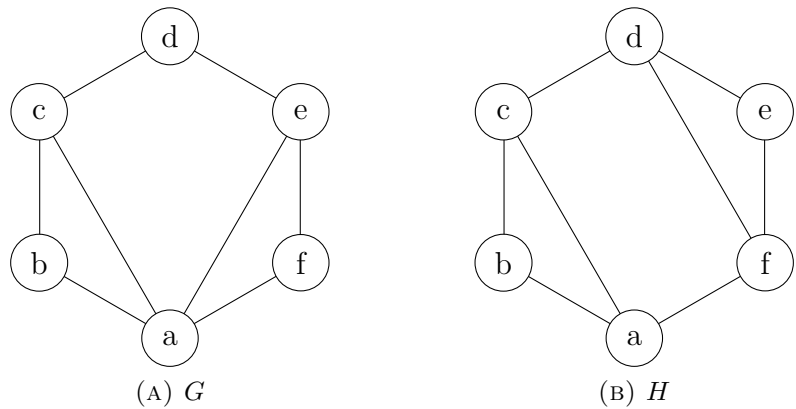
\begin{figure}[h!]
    
\subfloat[$G$]{    
\begin{tikzpicture}[rotate=30, transform shape]
\foreach \a in {6}{
\node[regular polygon, regular polygon sides=\a, minimum size=4cm, line width=.4pt, draw] at (\a*4,0) (A) {};
\draw[line width=.4pt] (A.corner 2) -- (A.corner 4) -- (A.corner 6);

    \node[circle, draw, minimum size=0.75cm, fill=white, rotate=-30] at (A.corner 1) {d};
    \node[circle, draw, minimum size=0.75cm, fill=white, rotate=-30] at (A.corner 2) {c};
    \node[circle, draw, minimum size=0.75cm, fill=white, rotate=-30] at (A.corner 3) {b};
    \node[circle, draw, minimum size=0.75cm, fill=white, rotate=-30] at (A.corner 4) {a};
    \node[circle, draw, minimum size=0.75cm, fill=white, rotate=-30] at (A.corner 5) {f};
    \node[circle, draw, minimum size=0.75cm, fill=white, rotate=-30] at (A.corner 6) {e};
}
\end{tikzpicture}
}%
\qquad \qquad
\subfloat[$H$]{
\begin{tikzpicture}[rotate=30, transform shape]
\foreach \a in {6}{
\node[regular polygon, regular polygon sides=\a, minimum size=4cm, line width=.4pt, draw] at (\a*4,0) (A) {};
\draw[line width=.4pt] (A.corner 2) -- (A.corner 4);
\draw[line width=.4pt] (A.corner 1) -- (A.corner 5);

    \node[circle, draw, minimum size=0.75cm, fill=white, rotate=-30] at (A.corner 1) {d};
    \node[circle, draw, minimum size=0.75cm, fill=white, rotate=-30] at (A.corner 2) {c};
    \node[circle, draw, minimum size=0.75cm, fill=white, rotate=-30] at (A.corner 3) {b};
    \node[circle, draw, minimum size=0.75cm, fill=white, rotate=-30] at (A.corner 4) {a};
    \node[circle, draw, minimum size=0.75cm, fill=white, rotate=-30] at (A.corner 5) {f};
    \node[circle, draw, minimum size=0.75cm, fill=white, rotate=-30] at (A.corner 6) {e};
    }
\end{tikzpicture}
}%
        \caption{Two non-isomorphic $2$-connected graphs $G$ and $H$ giving rise to unimodularly equivalent symmetric edge polytopes.}
        \label{fig:graphs}
    \end{figure}
    
    Let $G$ and $H$ be the $6$-vertex graphs depicted in \Cref{fig:graphs}. Both $G$ and $H$ are $2$-connected; moreover, since the vertex $a$ has degree $4$ in $G$ and all vertices have degree at most $3$ in $H$, the graphs $G$ and $H$ are not isomorphic. After matching the $i$-th letter of the English alphabet with the $i$-th coordinate of $\mathbb{R}^6$, consider the $5$-dimensional symmetric edge polytopes $\Pc_G$ and $\Pc_H$ in $\mathbb{R}^6$. Letting \[F = \begin{bmatrix}0 & -1 & -1 & 0 & 0 & 0\\0 & 1 & 0 & 0 & 0 & 0\\1 & 1 & 2 & 2 & 1 & 1\\0 & 0 & 0 & 0 & 1 & 0\\-1 & -1 & -1  & -1 & -1 & 0\\1 & 1 & 1 & 0 & 0 & 0\end{bmatrix} \in \mathrm{GL}_6(\ZZ),\]
    
one checks that the unimodular map $\psi_F: \mathbb{R}^6 \to \mathbb{R}^6$ sending $\xf$ to $F\xf$ transforms $\Pc_G$ into $\Pc_H$, and thus $\Pc_G$ and $\Pc_H$ are unimodularly equivalent.    
\end{remark}

\section{Facets of $\Pc_{\cM}$ and the Ehrhart theory of the polar polytope} \label{sec:polar}

After defining generalized symmetric edge polytopes and investigating their first structural properties, it is our next goal to find a combinatorial characterization of their facets. In order to achieve this, it is fruitful to focus on the Ehrhart theory of the polar polytope. Unless specified differently, in this section we will only consider \emph{simple} regular matroids of positive rank, so that by \Cref{thm:first properties}(v) the vertices of $\Pc_{\cM}$ will correspond to the columns of $\Mpm$ for any full-rank weakly unimodular matrix $M$ representing $\cM$.

Inspired by work of Beck and Zaslavsky \cite{BeZa}, we begin by providing a description of the lattice points in the $k$-th dilation of $\Pc^{\Delta}_{\cM}$. 

\begin{proposition} \label{prop:polar lattice points}
Let $k$ be a positive integer, $\cM$ be a simple regular matroid of rank $r > 0$ and $M$ be a full-rank weakly unimodular matrix representing $\cM$. Then the map
\[
\begin{split}
(k \cdot \Pc_M^{\Delta}) \cap \ZZ^r &\to \{(k+1)\text{-cuts of }M\}\\
\uf &\mapsto M^T{\uf}
\end{split}\]
is a bijection.
\end{proposition}
\begin{proof}

Let us first describe in more detail the polar polytope $\Pc_M^{\Delta}$. A facet description of $\Pc_M^{\Delta}$ is given by $\Mpm^T \uf \leq \mathbf{1}$, which in turn implies that
\begin{equation} \label{eq:polar description M}
    k\cdot\Pc_M^{\Delta} = \left\{\mathbf{u}\in \RR^r : -k\cdot\mathbf{1} \leq M^T\uf \leq k\cdot\mathbf{1}\right\},
\end{equation}
where the inequalities are meant to be taken componentwise. This implies that, if $\uf \in (k \cdot \Pc_M^{\Delta}) \cap \ZZ^r$, then $(M^T\uf)$ is an element of $\mathrm{row}(M) \cap \ZZ^n$ such that $|(M^T\uf)_i| \leq k$ for every $i \in [n]$. This means precisely that $M^T\uf$ is a $(k+1)$-cut of $M$.

Vice versa, let $\boldsymbol{\gamma}$ be a $(k+1)$-cut of $M$. Since $\boldsymbol{\gamma} \in \mathrm{row}(M)$, there exists $\uf \in \ZZ^r$ such that $M^T\uf = \boldsymbol{\gamma}$. Since $M$ is full-rank, the linear map $\ZZ^r \to \ZZ^n$ defined by $M^T$ is injective, and thus $\uf$ is uniquely determined. Since $\uf$ satisfies the inequalities in \eqref{eq:polar description M}, we have that $\uf$ is a lattice point of $k \cdot \Pc_M^{\Delta}$, and this finishes the proof.
\end{proof}

\begin{example}
Let $M$ be as in \Cref{ex:running}. Then the polar polytope $\Pc^{\Delta}_M$ is shown in \Cref{fig:dual gSEP example}. The lattice points of $\Pc^{\Delta}_M$ are obtained from the $2$-cuts in \Cref{ex:cuts and flows} by throwing away the last two coordinates.

    \begin{figure}[h!] 
    \includegraphics[scale=0.5]{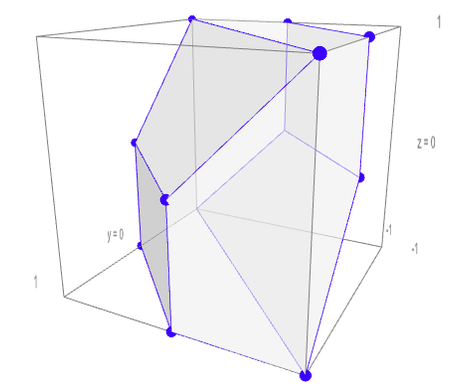}
        \caption{The polar polytope $\Pc^{\Delta}_M$, where $M$ is the matrix from \Cref{ex:running}. Plot generated by SageMath \cite{sagemath}.}
        \label{fig:dual gSEP example}
    \end{figure}
\end{example}

\begin{remark} \label{rem:lattice of cuts}
    A consequence of \Cref{prop:polar lattice points} is that the lattice of cuts $\Gamma(M)$ can be thought of as the union of the lattice points of $k \cdot \Pc_M^{\Delta}$ as $k$ varies in $\NN$ (with the convention that $0 \cdot \Pc_M^{\Delta} = \{\origin\}$). This gives us an interpretation of the lattice of cuts as a ``limit object''.
\end{remark}

It follows from the argument in \Cref{not:F_u} that, if $\uf$ is a lattice point of $\Pc_M^{\Delta}$, then \[F_{\uf} = \conv(\{M\ef_i \mid M\ef_i \cdot \uf = 1\} \cup \{-M\ef_{i} \mid M\ef_i \cdot \uf = -1\})\] is a face of $\Pc_M$ with supporting hyperplane \[H_{\uf} = \{\xf \in \RR^r \mid \xf \cdot \uf = 1\}\] (where we are using the fact that the columns of $\Mpm$ correspond to the vertices of $\Pc_M$). Since by \Cref{prop:polar lattice points} $\boldsymbol{\gamma} \coloneqq M^T\uf$ is a $2$-cut of $M$, we can rewrite $F_{\uf}$ in the following way:
\[F_{\uf} = \conv(\{M\ef_i \mid \gamma_i = 1\} \cup \{-M\ef_{i} \mid \gamma_i = -1\}).\]

In other words, the $2$-cut $\boldsymbol{\gamma} = M^T\uf$ acts as an indicator vector for $F_{\uf}$, in the following sense: the $i$-th entry of $\boldsymbol{\gamma}$ equals $+1$ (respectively, $-1$) if and only if the vertex $M\ef_i$ (respectively, $-M\ef_{i}$) belongs to $F_{\uf}$.

Next, we are going to define a partial order on $\{0, \pm1\}$-tuples that will enable us to give a first characterization of the facets of $\Pc_\cM$. 

\begin{definition} \label{def:partial order}
Let $\uf, \vf \in \{0, \pm1\}^m$. We will write that $\uf \preceq \vf$ if for every $i \in [m]$ it holds that $u_i = 0$ or $u_i = v_i$. Equivalently, $\preceq$ is the partial order induced componentwise by the relations ``$0 \prec +1$'', ``$0 \prec -1$'' and ``$+1$ and $-1$ are incomparable''.
\end{definition}

\begin{remark} \label{rem:face poset of the cross-polytope}
Note that $\{0, \pm 1\}^m$ equipped with the partial order from \Cref{def:partial order} is isomorphic to the face lattice of the $m$-dimensional cross-polytope: see for example \cite[Example 2.9]{AdamsReiner}. More specifically, the isomorphism maps $\boldsymbol{\gamma} \in \{0, \pm 1\}^m$ to the face obtained as $\conv(\{\ef_i \mid \gamma_i = 1\} \cup \{-\ef_i \mid \gamma_i = -1\})$. This foreshadows the upcoming characterizations of the facets of $\Pc_{\cM}$.
\end{remark}

A direct consequence of the definition of $\preceq$ is that $F_{\uf} \subseteq F_{\ufp}$ if and only if $M^T\uf \preceq M^T\ufp$. This immediately yields a first characterization of the facets of $\Pc_{\cM}$ when $\cM$ is simple.

\begin{corollary}\label{cor:facets via polarity}
Let $\cM$ be a simple regular matroid of positive rank and let $M$ be a full-rank weakly unimodular representation of $\cM$. Then the facets of $\Pc_M$ are the faces $F_{\uf}$ of $\Pc_M$ for which $M^T{\uf}$ is a $\preceq$-maximal $2$-cut of $M$.
\end{corollary}

The facet description in \Cref{cor:facets via polarity} is not completely satisfactory. Our next goal is to develop an alternate characterization that will be the ``right'' generalization of the description obtained by Higashitani, Jochemko and Micha{\l}ek \cite[Theorem 3.1]{HJM} for classical symmetric edge polytopes: see \Cref{rem:comparison with facets of classical SEPs} below for a more detailed discussion.

\begin{definition}
Let $\cM$ be a regular matroid of positive rank and let $M$ be a full-rank weakly unimodular representation of $\cM$. We will say that the cut $\boldsymbol{\gamma} \in \Gamma(M)$ is \emph{spanning} if the support of $\boldsymbol{\gamma}$ contains a basis of $\cM$.
\end{definition}

\begin{theorem} \label{thm:facets as spanning 2-cuts}
Let $\cM$ be a simple regular matroid of rank $r > 0$ and let $M$ be a full-rank weakly unimodular representation of $\cM$. Then the facets of $\Pc_M$ are the faces $F_{\uf}$ of $\Pc_M$ for which $M^T\uf$ is a spanning $2$-cut of $M$.
\end{theorem}
\begin{proof}
Let us recall once more that, since $\cM$ is simple, the vertices of $\Pc_M$ are in bijection with the columns of $\Mpm$ by \Cref{thm:first properties}(v).

Let us show that, if $\boldsymbol{\gamma} = M^T\uf$ is a spanning $2$-cut of $M$, then $F_{\uf}$ is a facet of $\Pc_M$. Since $\boldsymbol{\gamma}$ is spanning, by the discussion after \Cref{rem:lattice of cuts} we know that the face $F_{\uf}$ contains $r$ linearly independent vertices. Since $\origin \notin F_{\uf}$, such vertices are also affinely independent; but then, since $\dim(\Pc_M) = r$ by \Cref{thm:first properties}(ii), it follows that $F_{\uf}$ must be a facet.

Let us now prove that all facets of $\Pc_M$ arise in this fashion. Let $G$ be a facet of $\Pc_M$. Since $\dim(\Pc_M) = r$, the facet $G$ must contain $r$ linearly independent vertices $\vf_1, \ldots, \vf_r$; these will correspond to certain columns of $M$ or $-M$. If the $i$-th column of $M$ appears among the $\vf_j$'s, set $\gamma_i = 1$; if the $i$-th column of $-M$ does, set $\gamma_i = -1$. Possibly after some relabeling, we can assume without loss of generality that $\gamma_i \neq 0$ for every $i \in [r]$. By \Cref{lem:unique cut}, there exists a unique cut $\boldsymbol{\gamma}$ compatible with the above assignments; moreover, such a cut is spanning by construction. It only remains to show that $\boldsymbol{\gamma}$ is a $2$-cut. By polarity, the facet $G$ corresponds to a vertex $\uf'$ of the polar polytope $\Pc_M^{\Delta}$; it then follows from \Cref{prop:polar lattice points} that \begin{equation*}G = F_{\uf'} = \conv(\{M\ef_i \mid \gamma'_i = 1\} \cup \{-M\ef_{i} \mid \gamma'_i = -1\}),\end{equation*}
where $\boldsymbol{\gamma'} = M^T{\uf'}$ is a $2$-cut of $M$. Since $\boldsymbol{\gamma}$ and $\boldsymbol{\gamma'}$ coincide on a basis of $\cM$, it follows from the uniqueness of the cut in \Cref{lem:unique cut} that $\boldsymbol{\gamma} = \boldsymbol{\gamma'}$ and hence $\boldsymbol{\gamma}$ is a a $2$-cut, as desired.
\end{proof}

\begin{example}
Let $M$ be as in \Cref{ex:running}. Twelve of the seventeen $2$-cuts enumerated in \Cref{ex:cuts and flows} are spanning: these are $(1,0,0,-1,1)$, $(-1,0,0,1,-1)$, $(0,1,0,-1,1)$, $(0,-1,0,1,-1)$,  $(1,0,-1,0,1)$, $(-1,0,1,0,-1)$, $(0,1,-1,0,1)$, $(0,-1,1,0,-1)$, $(1,-1,1,-1,0)$, $(-1,1,-1,1,0)$, $(1,-1,-1,1,0)$, $(-1,1,1,-1,0)$.

Hence, $\Pc_M$ has twelve facets, and each of the spanning $2$-cuts serves as an indicator vector for one of them: for instance, the $2$-cut $(1,0,0,-1,1)$ corresponds to the facet obtained as the convex hull of $\ef_1$, $\ef_1+\ef_2+\ef_3$ and $\ef_1+\ef_2$ (respectively, the first, minus the fourth, and the fifth column of $M$).
\end{example}

\begin{remark} \label{rem:comparison with facets of classical SEPs}
Some words are needed in order to explain in which sense \Cref{thm:facets as spanning 2-cuts} generalizes the characterization of facets obtained by Higashitani, Jochemko and Micha{\l}ek for classical symmetric edge polytopes \cite[Theorem 3.1]{HJM}. If $G$ is a connected graph, facets of the symmetric edge polytope $\Pc_G$ were shown to be in bijection with integer vertex labelings such that
\begin{itemize}
\item[(i)] if $i$ and $j$ are adjacent in $G$, then their labels differ at most by one;
\item[(ii)] the subgraph of $G$ consisting of the edges $\{i, j\}$ whose vertex labels differ exactly by one contains a spanning tree of $G$.
\end{itemize}
(For the statement to be precise, one further needs to identify any two vertex labelings that differ by a fixed constant value on each vertex.) The first author, Delucchi and Micha{\l}ek observed in \cite[proof of Proposition 61]{DDM} that, after fixing an orientation of $G$, such a characterization is equivalent to asking for integer \emph{edge} labelings such that
\begin{itemize}
\item[(a)] each label is either $1$, $0$ or $-1$;
\item[(b)] the sum of the labels on each oriented cycle of $G$ is zero;
\item[(c)] the set of edges with nonzero labels contains a spanning tree of $G$. 
\end{itemize}
This last characterization corresponds to \Cref{thm:facets as spanning 2-cuts} in the special case when $\cM$ is the graphic matroid associated with $G$ and $M$ is the signed incidence matrix associated with the chosen orientation of $G$ (although, to be fully precise, such a matrix is not full-rank). Indeed, labeling the (oriented) edges of $G$ can be thought of as labeling the columns of the matrix $M$. More in detail, condition (b) can be expressed more succinctly by saying that the desired edge labelings are cuts of $M$, while conditions (a) and (c) further specify that they must be spanning $2$-cuts. 
\end{remark}

Comparing the facet characterization from \Cref{cor:facets via polarity} with the one found in \Cref{thm:facets as spanning 2-cuts} immediately yields the following corollary:

\begin{corollary}
Let $\cM$ be a simple regular matroid of positive rank and $M$ a full-rank weakly unimodular representation of $\cM$. Then a $2$-cut of $M$ is spanning if and only if it is $\preceq$-maximal.
\end{corollary}

The next result generalizes \cite[Proposition 61]{DDM}. We recall that a matroid is said to be \emph{bipartite} if all of its circuits have even cardinality.

\begin{proposition} \label{prop:bipartite facets}
Let $\cM$ be a simple regular bipartite matroid of rank $r > 0$ and let $M$ be a full-rank weakly unimodular representation of $\cM$.
Then the facets of $\Pc_{M}$ are in bijection with the nowhere-zero $2$-cuts of $M$.
\end{proposition}
\begin{proof}
By \Cref{thm:facets as spanning 2-cuts}, it is enough to prove that the spanning $2$-cuts of $M$ are exactly the nowhere-zero $2$-cuts of $M$. Clearly, every nowhere-zero $2$-cut must be spanning. 

For the reverse containment, let $n$ be the number of elements in the ground set of $\cM$. If $r=n$, then $\cM$ is the uniform matroid $\mathcal{U}_{n,n}$, the polytope $\Pc_{M}$ is unimodularly equivalent to the $n$-dimensional cross-polytope, and its $2^n$ facets correspond to the nowhere-zero elements of $\mathrm{row}(M) \cap \{0, \pm 1\}^n = \{0, \pm 1\}^n$ (see also \Cref{rem:face poset of the cross-polytope}).

Assume now that $r < n$. Let $\boldsymbol{\gamma}$ be a spanning $2$-cut of $M$ and assume without loss of generality that $\gamma_i \neq 0$ for every $i \in [r]$. Now pick any $j \in \{r+1, \ldots, n\}$ and consider the fundamental signed circuit $\overrightarrow{\Cc}(j, [r])$, whose support has even cardinality because of the bipartite assumption. Since $\boldsymbol{\gamma} \cdot \overrightarrow{\Cc}(j, [r]) = 0$, one has that $0$ is the sum of $\gamma_j$ and an odd number of elements in $\{+1, -1\}$. For parity reasons, it follows that $\gamma_j \neq 0$, which proves the claim. 
\end{proof}

\section{A regular unimodular triangulation for $\Pc_{\cM}$} \label{sec:gb}
It follows from a result of Ohsugi and Hibi \cite[Theorem 2.7]{OhsugiHibiCS} that the polytope $\Pc_{\cM}$ always admits a regular unimodular triangulation. The aim of this section is to find an explicit description generalizing what Higashitani, Jochemko and Micha{\l}ek found in the context of symmetric edge polytopes \cite[Proposition 3.8]{HJM}. Since the desired characterization involves signed circuits (see \Cref{def:signed circuits}), our results will be expressed in terms of a fixed full-rank weakly unimodular representation of the given (simple) regular matroid $\cM$.

If $M$ is a full-rank weakly unimodular matrix, then $\Mpm$ is as well. It will be useful to describe the signed circuits of the latter in terms of the former. To achieve this goal, we need to introduce some more notation.

\begin{notation} \label{notation:eta_I}
Let $J \subseteq [n]$. We will denote by $\eta_J$ the injective map $\ZZ^n \to \ZZ^{2n}$ sending $(\lambda_1, \ldots, \lambda_n)$ to $(\tlambda_1, \ldots, \tlambda_{2n})$, where for every $i \in [n]$
\[
\begin{aligned}
\tlambda_i \coloneqq 
{\begin{cases}
0 & i\in J\\
\lambda_i & i\notin J
\end{cases}} & \quad\text{and} &
\tlambda_{n+i} \coloneqq 
{\begin{cases}
-\lambda_i & i\in J\\
0 & i\notin J.
\end{cases}} 
\end{aligned}
\]

Basically, given an integer vector, the map $\eta_J$ changes the sign of the entries indexed by an element of $J$, and then moves them to the second half of an integer vector twice as long. We will sometimes refer to this operation as a \emph{promotion}. Note that $\eta_J$ restricts to a map $\{0, \pm1\}^n \to \{0, \pm1\}^{2n}$.
\end{notation}

\begin{remark} \label{rem:image of eta_I}
Note that $\btlambda \in \ZZ^{2n}$ is in the image of $\eta_J$ precisely when $\tlambda_i = 0$ for every $i \in J$ and $\tlambda_{n+i} = 0$ for every $i \notin J$. In particular, if the support of $\btmu \in \ZZ^{2n}$ is contained in the support of $\btlambda \in \mathrm{im}(\eta_J)$, then $\btmu \in \mathrm{im}(\eta_J)$.
\end{remark}

\begin{lemma} \label{lem:eta_circuits}
Let $J \subseteq [n]$ and let $M \in \ZZ^{r \times n}$ (where $0 < r \leq n$) be a full-rank weakly unimodular integer matrix. Then $\blambda \in \sgncirc(M)$ if and only if $\eta_J(\blambda) \in \sgncirc([M\ \!|\ \!-M])$.
\end{lemma}
\begin{proof}
Let $\blambda \in \ZZ^n$. By construction, one has that
\begin{equation*} \label{eq:lambda and lambda_tilde}
M\blambda = \sum_{i \in [n]}\lambda_i (M\ef_i) = \sum_{i \in [n]}(\tlambda_i - \tlambda_{n+i}) (M\ef_i) = \sum_{i \in [n]}\tlambda_i (M\ef_i) + \sum_{i \in [n]}\tlambda_{n+i}(-M\ef_{i}) = \Mpm\eta_J(\blambda).
\end{equation*}
In particular, $\blambda \in \ker(M)$ if and only if $\eta_J(\blambda) \in \ker\left(\Mpm\right)$, and $\blambda$ is a $2$-flow if and only if $\eta_J(\blambda)$ is. To prove the claim, we still need to show that $\supp(\blambda)$ is a circuit of $M$ if and only if $\supp(\eta_J(\blambda))$ is a circuit of $\Mpm$.

The definition of $\eta_J$ implies immediately that, if $\supp(\blambda)$ is not minimally dependent, then $\supp(\eta_J(\blambda))$ is not minimally dependent either. Conversely, assume $\supp(\eta_J(\blambda))$ is not minimally dependent. Then there exists $\btmu \in \ZZ^{2n}$ such that $\Mpm\btmu = \origin$ and $\supp(\btmu) \subseteq \supp(\eta_J(\blambda))$. By \Cref{rem:image of eta_I}, $\btmu$ belongs to the image of $\eta_J$ and hence $\supp(\blambda)$ is not minimally dependent, since $\supp(\eta_J^{-1}(\btmu)) \subseteq \supp(\blambda)$. This proves the claim.
\end{proof}

Provided that $M$ does not contain any zero column, the signed circuits of $\Mpm$ come in two flavors: on the one hand, we have the ones of the form $\pm(\ef_i + \ef_{n+i})$ (reflecting the relation $M\ef_i + (-M)\ef_{i} = \origin$), while on the other hand we have those obtained by promoting a signed circuit of $M$. This is the content of the technical lemma below.

\begin{lemma} \label{lem:circuits of Mpm}
Assume $M \in \ZZ^{r \times n}$ (where $0 < r \leq n$) is a full-rank weakly unimodular matrix not containing any zero column. Then
\begin{equation} \label{eq:circuits of Mpm}
    \overrightarrow{\Ck}([M\ \!|\ \!-M]) = \{\pm(\ef_i + \ef_{n+i}) : i \in [n]\} \cup \bigcup_{J \subseteq [n]}\eta_J(\overrightarrow{\Ck}(M)).
\end{equation}
\end{lemma}
\begin{proof}
Let us first prove that the right hand side of \eqref{eq:circuits of Mpm} consists of signed circuits of $\Mpm$. This is clear for $\pm(\ef_i + \ef_{n+i})$, since \[\Mpm \ef_i + \Mpm \ef_{n+i} = M\ef_i + (-M)\ef_{i} = \origin\] and $M$ does not contain any zero column by hypothesis. Moreover, by \Cref{lem:eta_circuits}, $\eta_J(\blambda)$ is a signed circuit of $\Mpm$ for every choice of $J \subseteq [n]$ and $\blambda \in \overrightarrow{\Ck}(M)$.

Let us now prove that \emph{every} signed circuit of $\Mpm$ arises as in the right-hand side of \eqref{eq:circuits of Mpm}. Let $\btlambda = (\tlambda_1, \ldots, \tlambda_{2n}) \in \{0, \pm 1\}^{2n}$ be a signed circuit of $\Mpm$. If there exists $i \in [n]$ such that $\tlambda_i\tlambda_{n+i} \neq 0$ then, by support minimality, $\btlambda$ must be equal to $\ef_i + \ef_{n+i}$ up to sign. Assume then that $\tlambda_i\tlambda_{n+i} = 0$ for every $i \in [n]$, and let $J \coloneqq \{i \in [n] : \tlambda_{n+i} \neq 0\}$. By \Cref{rem:image of eta_I}, one has that $\btlambda = \eta_J(\blambda)$ for some $\blambda \in \{0, \pm 1\}^n$; moreover, by \Cref{lem:eta_circuits}, such $\blambda$ is a signed circuit of $M$. This finishes the proof.
\end{proof}

We remark that the unimodularity assumption is not really crucial for Lemmas \ref{lem:eta_circuits} and \ref{lem:circuits of Mpm}: one could prove similar statements by substituting ``signed circuits'' with ``circuits'' (in the toric ideal meaning). However, since in this paper we are reserving the word ``circuit'' for its matroidal meaning, we did not want to confuse the reader unnecessarily.

Before moving on, we need to introduce some notation about toric ideals naturally arising in this context.

\begin{notation} \label{notation:simple matroid toric}
Let $\cM$ be a \emph{simple} regular matroid of rank $r > 0$ on $n$ elements and let $M$ be a full-rank weakly unimodular representation of $\cM$. Then, by \Cref{thm:first properties}(iv)--(v), the lattice points of $\Pc_M$ are the columns of $\Mpm$ and the origin $\origin$. We will denote by $I_{\Pc_{M}}$ the toric ideal associated with the polytope $\Pc_{M}$, i.e., the one obtained as the kernel of the map
\begin{equation*}
    \begin{aligned}
        K[x_1, \ldots, x_n, x_{-1}, \ldots, x_{-n}, z] &\to K[t_1^{\pm 1}, \ldots, t_r^{\pm 1}, s]\\
        x_i &\mapsto \mathbf{t}^{M\ef_i}s\\
        x_{-i} &\mapsto \mathbf{t}^{-M\ef_{i}}s\\
        z &\mapsto s
    \end{aligned}
\end{equation*}
and by $I_{[M\ \!|\ \!-M]}$ the toric ideal obtained as the kernel of the map 
\begin{equation*}
    \begin{aligned}
        K[x_1, \ldots, x_n, x_{-1}, \ldots, x_{-n}] &\to K[t_1^{\pm 1}, \ldots, t_r^{\pm 1}]\\
        x_i &\mapsto \mathbf{t}^{M\ef_i}\\
        x_{-i} &\mapsto \mathbf{t}^{-M\ef_{i}},
    \end{aligned}
\end{equation*}
where $K$ is a field.
\end{notation}

We immediately obtain the following corollary of \Cref{lem:toric homogenization}:

\begin{corollary} \label{cor:homogenization of I_Mpm}
Let $\cM$ be a simple regular matroid of rank $r > 0$ and let $M \in \ZZ^{r \times n}$ be a full-rank weakly unimodular matrix representing $\cM$. Then the ideal $I_{\Pc_M}$ is the homogenization of $I_{[M\ \!|\ \!-M]}$ with respect to the variable $z$. In particular, the (irreducible) projective variety $V(I_{\Pc_M})$ is the projective closure of the (irreducible) affine variety $V(I_{[M\ \!|\ \!-M]})$.
\end{corollary}

\Cref{lem:circuits of Mpm} and \Cref{rem:unimodular toric} imply the following description for the universal Gr\"obner basis of the toric ideal $I_{[M\ \!|\ \!-M]}$ when $M$ is a full-rank weakly unimodular matrix not containing any zero column:

\begin{corollary} \label{cor:UGB for I_Mpm}
Let $M \in \ZZ^{r \times n}$ (where $0 < r \leq n$) be a full-rank weakly unimodular matrix without any zero column. Then the set of signed circuits, the universal Gr\"obner basis and the Graver basis of $I_{[M\ \!|\ \!-M]}$ all coincide and consist of the following binomials:
\begin{itemize}
    \item $x_ix_{-i} - 1$ for every $i \in [n]$;
    \item $\xf^{\eta_J(\blambda)^+} - \xf^{\eta_J(\blambda)^-}$ for every $J \subseteq [n]$, $\blambda \in \overrightarrow{\Ck}(M)$ such that $|\eta_J(\blambda)^+| \geq |\eta_J(\blambda)^-|$.
\end{itemize}
(With a slight abuse of notation, we identify those binomials that differ only up to a global sign.) 
\end{corollary}

\begin{example}
    Let $M$ be as in \Cref{ex:running} and \Cref{ex:cuts and flows}. The Graver basis of $\Mpm$ contains 37 binomials: $5$ of the form $x_ix_{-i}-1$, $16$ arising from the promotions of $x_1x_2x_3x_4-1$, $8$ from the promotions of $x_1x_2-x_5$ and another $8$ from the promotions of $x_3x_4x_5-1$. For instance, the promotions of $\blambda = (1,1,0,0,-1) \in \sgncirc(M)$ (corresponding to the binomial $x_1x_2 - x_5$) give rise to the following eight signed circuits of $\Mpm$: $x_1x_2-x_5$, $x_2-x_{-1}x_5$, $x_1 - x_{-2}x_5$, $x_1x_2x_{-5}-1$, $1 - x_{-1}x_{-2}x_5$, $x_2x_{-5}-x_{-1}$, $x_1x_{-5} - x_{-2}$, $x_{-5} - x_{-1}x_{-2}$. Technically, the statement of \Cref{cor:UGB for I_Mpm} only asks for the promotions such that $|\eta_J(\blambda)^+| \geq |\eta_J(\blambda)^-|$; however, when this is not the case, we just take $-\blambda$ instead of $\blambda$, and the global count is not affected. 
    
\end{example}

The next proposition, reminiscent of the results in \cite[Section 2]{OhsugiHibiCS}, proves the existence of a regular unimodular triangulation for $I_{\Pc_{\cM}}$ and serves as a first step towards an explicit description. For the correspondence between regular unimodular triangulations and squarefree initial ideals, we refer the reader to \cite[Chapter 8]{Stu} and \cite[Section 2.4]{UnimodularSurvey}.

\begin{proposition} \label{prop:squarefree initial ideal}
Let $\cM$ be a simple regular matroid of rank $r > 0$ and let $M \in \ZZ^{r \times n}$ be a full-rank weakly unimodular representation of $\cM$. Denote by $S$ the polynomial ring $K[x_1, \ldots, x_n, x_{-1}, \ldots, x_{-n}]$. Let $<$ be a graded monomial order of $S$ and let $<_h$ be any monomial order of $S[z]$ with the property that $\init_{<_h}f^h = \init_{<}f$ for every $f \in S$. Then the toric ideal $I_{\Pc_{M}}$ has a squarefree initial ideal with respect to $<_h$.
\end{proposition}

\begin{remark}
    A concrete choice for $<_h$ as in \Cref{prop:squarefree initial ideal} (and later \Cref{thm:gb}) is any degrevlex order of $S[z]$ such that $z <_h v$ for every variable $v$ in $S$.
\end{remark}

\begin{proof}[Proof of \Cref{prop:squarefree initial ideal}]
By \Cref{cor:homogenization of I_Mpm}, the toric ideal $I_{\Pc_M}$ is the homogenization of $I_{[M\ \!|\ \!-M]}$ with respect to the variable $z$. In order to find a Gr\"obner basis for $I_{\Pc_M}$, by \cite[Theorem 8.4.4]{CLO} it is then enough to homogenize a set of polynomials forming a Gr\"obner basis for $I_{[M\ \!|\ \!-M]}$ with respect to a graded monomial order. The universal Gr\"obner basis of $I_{[M\ \!|\ \!-M]}$ is described in \Cref{cor:UGB for I_Mpm}; since by definition signed circuits have coefficients in $\{0, \pm 1\}$, the claim follows.
\end{proof}

We are finally able to generalize the Gr\"obner basis description obtained by Higashitani, Jochemko and Micha{\l}ek for the usual symmetric edge polytopes.

\begin{theorem} \label{thm:gb}
    Let $\cM$ be a simple regular matroid and let $M \in \ZZ^{r \times n}$ be a full-rank weakly unimodular representation of $\cM$. Let $S[z]$, $<$ and $<_h$ be as in \Cref{prop:squarefree initial ideal}. Then the polynomials
    \begin{enumerate}
        \item[(i)] $x_ix_{-i} - z^2$ for every $i \in [n]$;
        \item[(ii)] $\xf^{\eta_J(\blambda)^+} - \xf^{\eta_J(\blambda)^-}z$ for every $J \subseteq [n]$, $\blambda \in \overrightarrow{\Ck}(M)$ such that $|\eta_J(\blambda)^+| = |\eta_J(\blambda)^-|+1$;
        \item[(iii)] $\xf^{\eta_J(\blambda)^+} - \xf^{\eta_J(\blambda)^-}$ for every $J \subseteq [n]$, $\blambda \in \overrightarrow{\Ck}(M)$ such that $|\eta_J(\blambda)^+| = |\eta_J(\blambda)^-|$
    \end{enumerate}
    form a Gr\"obner basis for $I_{\Pc_M}$ with respect to $<_h$.
\end{theorem}

\begin{remark}
The binomials of types (ii) and (iii) in \Cref{thm:gb} come from considering those signed circuits of $\Mpm$ where $1$'s and $-1$'s are ``as balanced as possible''; this is exactly what happens for classical symmetric edge polytopes, recalling that both orientations for every edge of the original undirected graph are available in that setting. 
\end{remark}

\begin{proof}[Proof of \Cref{thm:gb}]
    By the proof of \Cref{prop:squarefree initial ideal}, we know that homogenizing the polynomials of \Cref{cor:UGB for I_Mpm} with respect to the variable $z$ yields a Gr\"obner basis for $I_{\Pc_{M}}$ with respect to $<_h$. 
    For every $i \in [n]$, the homogenization of $x_ix_{-i} - 1$ gives us one of the polynomials of type (i). Now let $J \subseteq [n]$, $\blambda \in \overrightarrow{\Ck}(M)$ and set $k \coloneqq |\eta_J(\blambda)^+| - |\eta_J(\blambda)^-|$. After possibly swapping $\blambda$ with $-\blambda$, we can assume without loss of generality that $k \geq 0$.
    
    If $k = 0$ or $k = 1$, the homogenization of $\xf^{\eta_J(\blambda)^+} - \xf^{\eta_J(\blambda)^-}$ yields one of the binomials of type (iii) or (ii) in the list. It is then enough to show that the homogenization of $\xf^{\eta_J(\blambda)^+} - \xf^{\eta_J(\blambda)^-}$ is redundant when $k \geq 2$. Consider such a polynomial. There exists $j \in [n]$ such that either $\eta_J(\blambda)_j = 1$ or $\eta_J(\blambda)_{n+j} = 1$. If $\eta_J(\blambda)_j = 1$, one has that \[x_j \cdot \xf^{\eta_{J \cup \{j\}}(\blambda)^+} = \xf^{\eta_J(\blambda)^+} \ \text{ and } \ x_{-j} \cdot \xf^{\eta_{J}(\blambda)^-} = \xf^{\eta_{J \cup \{j\}}(\blambda)^-}\] and we can write 
    
    \[\begin{split}\xf^{\eta_J(\blambda)^+} - \xf^{\eta_J(\blambda)^-}z^k &= \xf^{\eta_{J}(\blambda)^+} - \xf^{\eta_{J}(\blambda)^-}z^k + x_jx_{-j}\xf^{\eta_{J}(\blambda)^-}z^{k-2} - x_jx_{-j}\xf^{\eta_{J}(\blambda)^-}z^{k-2}  \\
    &= x_{j} \cdot \xf^{\eta_{J\cup \{j\}}(\blambda)^+} - \xf^{\eta_{J}(\blambda)^-}z^k + x_jx_{-j}\xf^{\eta_{J}(\blambda)^-}z^{k-2} - x_j \cdot \xf^{\eta_{J \cup \{j\}}(\blambda)^-}z^{k-2}\\
    &= x_{j} \cdot (\xf^{\eta_{J\cup \{j\}}(\blambda)^+} - \xf^{\eta_{J \cup \{j\}}(\blambda)^-}z^{k-2}) + \xf^{\eta_J(\blambda)^-}z^{k-2}(x_jx_{-j}-z^2).\end{split}\]
    
    If instead $\eta_J(\blambda)_{n+j} = 1$, one has that \[x_{-j} \cdot \xf^{\eta_{J \setminus \{j\}}(\blambda)^+} = \xf^{\eta_J(\blambda)^+} \ \text{ and } \ x_{j} \cdot \xf^{\eta_{J}(\blambda)^-} = \xf^{\eta_{J \setminus \{j\}}(\blambda)^-}\] and an analogous computation leads to
    \[\xf^{\eta_J(\blambda)^+} - \xf^{\eta_J(\blambda)^-}z^k = x_{-j} \cdot (\xf^{\eta_{J\setminus \{j\}}(\blambda)^+} - \xf^{\eta_{J \setminus \{j\}}(\blambda)^-}z^{k-2}) + \xf^{\eta_{J}(\blambda)^-}z^{k-2}(x_jx_{-j}-z^2).\]
    Iterating this procedure as many times as possible yields the claim.
\end{proof}

\begin{example} \label{ex:triangulation}
    Let $M$ be as in \Cref{ex:running} and \Cref{ex:cuts and flows}, and pick as a term order the degree reverse lexicographic order with $x_{1} > x_{-1} > x_{2} > x_{-2} > x_{3} > x_{-3} > x_{4} > x_{-4} > x_{5} > x_{-5} > z$. Then, by \Cref{thm:gb}, there is a Gr\"obner basis for $I_{\Pc_M}$ consisting of the following binomials (where we underline the leading term):
    \begin{itemize}
    \item $\underline{x_1x_{-1}}-z^2$, $\underline{x_2x_{-2}}-z^2$, $\underline{x_3x_{-3}}-z^2$, $\underline{x_4x_{-4}}-z^2$, $\underline{x_5x_{-5}}-z^2$ 
    \item $\underline{x_1x_2}-x_5z$, $\underline{x_{-1}x_{-2}}-x_{-5}z$, $\underline{x_{-1}x_5}-x_2z$, $\underline{x_1x_{-5}}-x_{-2}z$, $\underline{x_{-2}x_5}-x_1z$, $\underline{x_2x_{-5}}-x_{-1}z$
    \item $\underline{x_3x_4}-x_{-5}z$, $\underline{x_{-3}x_{-4}}-x_5z$, $\underline{x_3x_5}-x_{-4}z$, $\underline{x_{-3}x_{-5}}-x_4z$, $\underline{x_4x_5}-x_{-3}z$, $\underline{x_{-4}x_{-5}}-x_3z$
    \item $\underline{x_1x_2}-x_{-3}x_{-4}$, $\underline{x_{-1}x_{-2}}-x_3x_4$, $\underline{x_1x_3}-x_{-2}x_{-4}$, $\underline{x_{-1}x_{-3}}-x_2x_4$, $-\underline{x_{-2}x_{-3}}+x_1x_4$, $-\underline{x_2x_3}+x_{-1}x_{-4}$.
    \end{itemize}
    Note that this Gr\"obner basis is not reduced, as the monomials $x_1x_2$ and $x_{-1}x_{-2}$ are both featured twice as the leading term of a binomial. The associated triangulation has sixteen facets and is shown in \Cref{fig:triangulation}.
    \begin{figure}[h!] 
        \includegraphics[scale=0.5]{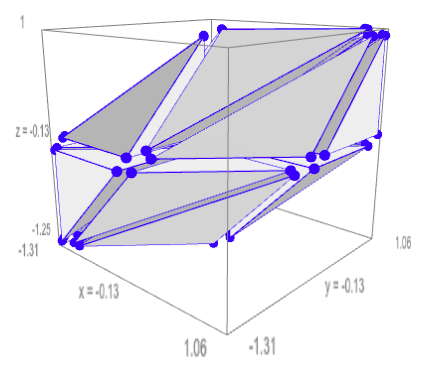}
        \caption{The triangulation described in \Cref{ex:triangulation}. Plot generated by SageMath \cite{sagemath}.}
        \label{fig:triangulation}
    \end{figure}
\end{example}

Finally, the $\gamma$-polynomial of a symmetric edge polytope has been the object of much recent work after Ohsugi and Tsuchiya conjectured the nonnegativity of its coefficients in \cite{OhsugiTsuchiya}. We wish to conclude the present article by extending to the matroidal setting a characterization of $\gamma_1$ which appeared independently in \cite{DJKKV} and \cite{KaTo}.

\begin{corollary}
Let $\cM$ be a simple regular matroid of positive rank. Then $\gamma_1(\Pc_{\cM}) = 2 \cdot \mathrm{rk}(\cM^*)$. In particular, $\gamma_1(\Pc_{\cM})$ is nonnegative.
\end{corollary}
\begin{proof}

In what follows, let $E$ be the ground set of the matroid $\cM$. By \Cref{prop:squarefree initial ideal}, the polytope $\Pc_{\cM}$ admits a (regular) unimodular triangulation $\Delta_{<}$, and hence the $h^*$-polynomial of $\Pc_{\cM}$ and the $h$-polynomial of $\Delta_{<}$ coincide. Then
    \begin{align*}
        \gamma_1(\Pc_{\cM}) &= h^\ast_1(\Pc_{\cM})- \mathrm{rk}(\cM) & \text{by \Cref{thm:first properties}(ii)}\\
        &= h_1(\Delta_{<})-\mathrm{rk}(\cM) & \text{by \Cref{prop:squarefree initial ideal}}\\
        &= (f_0(\Delta_{<})-\mathrm{rk}(\cM)) - \mathrm{rk}(\cM) & \text{by definition of $h_1$}\\&=2 \cdot (|E|-\mathrm{rk}(\cM)) & \text{since $\Delta_{<}$ has $2 \cdot |E|$ vertices}\\&=2 \cdot \mathrm{rk}(\cM^*).
    \end{align*}    
     
\end{proof}

\section{Future directions} \label{sec:future}
We conclude the present paper with some questions.
\begin{question}
Are generalized symmetric edge polytopes $\gamma$-positive? A positive answer would settle the conjecture by Ohsugi and Tsuchiya on symmetric edge polytopes \cite{OhsugiTsuchiya}.
More modestly, one could try to prove or disprove that $\gamma_2$ is always nonnegative, analogously to the classical symmetric edge polytope case treated in \cite{DJKKV}.
\end{question}

\begin{question}
How do properties of the generalized symmetric edge polytope (e.g., its $h^*$-vector) change under operations on the associated matroid? Is there any way to use Seymour's characterization of regular matroids via $1-$, $2-$ and $3-$sums \cite{Seymour}?
\end{question}

\begin{question}
Can one determine a formula for the $h^*$-vector of generalized symmetric edge polytopes analogous to the one found by K\'alm\'an and T\'othm\'er\'esz in \cite{KT_hstar}?
\end{question}

\begin{question}
Are there ``nice'' classes of regular matroids for which the $h^\ast$-polynomial of the associated generalized symmetric edge polytope is real-rooted?
\end{question}

\begin{question}
To which extent can the formulas from \cite[Propositions 64 and 65]{DDM} be generalized to the matroidal setting?
\end{question}

\subsection*{Acknowledgements}
We wish to thank Emanuele Delucchi, Akihiro Higashitani, Hidefumi Ohsugi and Lorenzo Venturello for useful comments and discussions at various stages of this project. We are grateful to Matthias Walter for helping us with using the Combinatorial Matrix Recognition Library (currently available at \url{http://discopt.github.io/cmr/}); in particular, in our computations we made use of the total unimodularity test described in \cite{WalterT13}. We also acknowledge the use of the Macaulay2 \cite{M2} package \texttt{Matroids} \cite{Chen} by Justin Chen. Finally, we thank Marco Caselli and Lorenzo Venturello for their help with coding in SageMath \cite{sagemath}.

\bibliographystyle{alpha} 
\bibliography{bibliography}
\end{document}